\documentclass[reqno]{amsart}
\usepackage[all]{xy}

\usepackage{amsmath}
\usepackage{amscd}
\usepackage{amsthm}
\usepackage{amssymb}
\usepackage{amsfonts}
\usepackage{epsfig}
\usepackage{enumerate}
\usepackage{graphicx}
\usepackage{graphicx, color}
\usepackage[mathcal]{euscript}
\usepackage{mathtools}
\usepackage{mathrsfs}

\usepackage{tikz}
\usetikzlibrary{arrows,decorations.markings}
\usetikzlibrary{plotmarks}
\usepackage{subfig}

\usepackage{indentfirst} 
\usepackage{xcolor}
\usetikzlibrary{calc, positioning} 

          \theoremstyle{definition}
          \newtheorem{theorem}{Theorem}[section]
          \newtheorem{prop}[theorem]{Proposition}
          
          \newtheorem{lem}[theorem]{Lemma}

          \newtheorem{cor}[theorem]{Corollary}

          \newtheorem{defn}[theorem]{Definition}

          \newtheorem{exam}[theorem]{Example}

          \newtheorem{rem}[theorem]{Remark}

          \newtheorem{ques}[theorem]{Question}
\numberwithin{equation}{section}

          \newcommand{\nc}{\newcommand}
          \nc{\DMO}{\DeclareMathOperator}	
          \nc{\commentout}[1]{}

          \nc{\newnotation}{\nomenclature}
          \nc{\wrap}{\cW}
          \nc{\Tw}{\mathsf{Tw}}
          \nc{\loc}{\mathsf{Loc}}
          \nc{\Top}{Top}
          \nc{\emb}{\mathsf{emb}}
          \nc{\ind}{\mathsf{Ind}}
          \nc{\Ind}{\mathsf{Ind}}
          \nc{\Loc}{\mathsf{Loc}}
          \nc{\Cob}{\mathsf{Cob}}
          \nc{\mul}{\mathsf{Mul}}
          \nc{\fat}{\mathsf{fat}}
          \nc{\cob}{\mathsf{Cob}}
          \nc{\coh}{\mathsf{Coh}}
          \nc{\Liouaut}{\Aut_{\mathsf{Liou}}}
          \nc{\idem}{\mathsf{Idem}}
          \nc{\sets}{\mathsf{Sets}}
          \nc{\near}{\mathsf{near}}
          \nc{\sing}{\mathsf{Sing}}
          \nc{\Sing}{\mathsf{Sing}}
          \nc{\perf}{\mathsf{Perf}}
          \nc{\block}{\mathsf{block}}
          \nc{\ssets}{\mathsf{sSets}}
          \nc{\cmpct}{\mathsf{cmpct}}
          \nc{\compact}{\mathsf{cmpct}}
          \nc{\pwrap}{\mathsf{PWrap}}
          \nc{\coder}{\mathsf{Coder}}
          \nc{\bimod}{\mathsf{Bimod}}
          \nc{\grmod}{\mathsf{GrMod}}
          \nc{\Morita}{\mathsf{Morita}}
          \nc{\morita}{\mathsf{Morita}}
          \nc{\spaces}{\mathsf{Spaces}}
          \nc{\pwrms}{\mathsf{PWrFuk}_{M,S}}
          \nc{\pwrmf}{\mathsf{PWrFuk}_{M,F}}
          \nc{\pwrapmf}{\mathsf{PWrFuk}_{M,F}}
          \nc{\fuk}{\mathsf{Fukaya}}
          \nc{\infwr}{\mathsf{InfWr}}
          \nc{\fukaya}{\mathsf{Fukaya}}
          \nc{\autml}{\mathsf{Aut}_{M,\Lambda}}
          \nc{\fukml}{\mathsf{Fukaya}_{M,\Lambda}}
          \nc{\fukmle}{\mathsf{Fukaya}_{M,\Lambda,\epsilon}}
          \nc{\fukmod}{\wrfukcompact(M)\modules}
          \nc{\lag}{\mathsf{Lag}}
          \nc{\lagm}{\lag_M}
          \nc{\lago}{\lag^o}
          \nc{\lagml}{\lag_{M,\Lambda}} 
          \nc{\lagmle}{\lag_{M,\Lambda,\epsilon}}
          \nc{\Fun}{\mathsf{Fun}}
          \nc{\fun}{\mathsf{Fun}}
          \nc{\vect}{\mathsf{Vect}}
          \nc{\chain}{\mathsf{Chain}}
          \nc{\chainn}{Chain}
          \nc{\wrfuk}{\mathsf{WrFukaya}}
          \nc{\wrfukcompact}{\mathsf{WrFukaya}_{\mathsf{cmpct}}}
          \nc{\pwrfuk}{\mathsf{PWrFukaya}}
          \nc{\inffuk}{\mathsf{InfFuk}}
          \nc{\pwrfukml}{\mathsf{PWrFukaya}_{M,\Lambda}}
          \nc{\inffukml}{\mathsf{InfFuk}_{M,\Lambda}}
          \nc{\nattrans}{\mathsf{NatTrans}}
          \nc{\corres}{\mathsf{Corres}}
          \nc{\fukep}{\fukaya_\Lambda(M,\epsilon)}
          \nc{\fukepop}{\fukaya_\Lambda(M,\epsilon)^{\op}}
          \nc{\lagep}{\lag_\Lambda(M,\epsilon)}
          \DMO{\cyl}{cyl} 
          \nc{\dbcoh}{D^b\mathsf{Coh}}
          \nc{\corr}{\mathsf{Corr}}
          \nc{\Liouauto}{{\Aut^o}}
          \nc{\Liouautb}{\Aut^{b}}
          \nc{\Liouautgr}{{\Aut^{gr}}}
          \nc{\Liouautgrb}{\Aut^{gr,b}}
          \nc{\Fuk}{\mathsf{Fuk}}

          \DMO{\im}{im}
          \DMO{\ev}{ev}
          \DMO{\stable}{Ex}
          \DMO{\inj}{inj}
          \DMO{\fib}{fib}
          \DMO{\conf}{Conf}
          \DMO{\chains}{Chains}
          \DMO{\cochains}{Cochains}
          \DMO{\cone}{Cone}
          \DMO{\Map}{Map}
          \DMO{\ran}{Ran}
          \DMO{\rot}{Rot}
          \DMO{\leg}{Leg}
          \DMO{\imm}{imm}
          \DMO{\adj}{adj}
          \DMO{\symp}{Symp}
          \DMO{\tree}{Tree}
          \DMO{\cube}{Cube}
          \DMO{\deep}{deep}
          \DMO{\back}{back}
          \DMO{\Hoch}{Hoch}
          \DMO{\front}{front}
          \DMO{\flow}{Flow}
          \DMO{\floer}{Floer}
          \DMO{\Maps}{Maps}
          \DMO{\exact}{exact}
          \DMO{\excess}{Excess}
          \DMO{\Decomp}{Decomp}
          \DMO{\decomp}{Decomp}
          \DMO{\collar}{collar}
          \DMO{\yoneda}{Yoneda}
          \DMO{\hamspace}{Ham}
          \DMO{\sympspace}{Symp}
          \DMO{\holomaps}{Holomaps}
          \DMO{\comp}{Comp}
          \DMO{\crit}{Crit}
          \DMO{\test}{{test}}
          \DMO{\sign}{sign}
          \DMO{\topp}{top}
          \DMO{\indx}{Index}
          \DMO{\Break}{Break} 
          \DMO{\zero}{zero} 
          \DMO{\ob}{Ob}
          \DMO{\gr}{Gr} 
          \DMO{\Gr}{Gr} 
          \DMO{\cl}{Cl} 
          \DMO{\grlag}{GrLag}
          \DMO{\Pin}{Pin}
          \DMO{\Graph}{Graph}
          \DMO{\pin}{Pin}
          \DMO{\gap}{Gap}
          \DMO{\Ex}{Ex}
          \DMO{\id}{id}
          \DMO{\End}{End}
          \DMO{\sym}{Sym}
          \DMO{\aut}{Aut}
          \DMO{\Aut}{Aut}
          \DMO{\haut}{hAut}
          \DMO{\hAut}{hAut}
          \DMO{\DK}{DK} 
          \DMO{\poly}{poly} 
          \DMO{\diff}{Diff}
          \DMO{\coll}{coll}
          \DMO{\dist}{dist} 
          \DMO{\coker}{coker} 
          \nc{\kernel}{\ker} 
          \DMO{\sspan}{span}
          \DMO{\hocolim}{hocolim}	
          \DMO{\holim}{holim}
          \DMO{\sk}{sk}
          \DMO{\Symp}{Symp}

          \DMO{\ho}{ho}
          \DMO{\fin}{fin}
          \DMO{\tor}{Tor}
          \DMO{\ext}{Ext}
          \DMO{\ret}{Ret}
          \DMO{\ham}{Ham}
          \DMO{\con}{con}
          \DMO{\leaf}{leaf}
          \DMO{\supp}{supp}
          \DMO{\edge}{edge}
          \DMO{\colim}{colim}
          \DMO{\edges}{edges}
          \DMO{\Image}{image}
          \DMO{\roots}{roots}
          \DMO{\height}{height}
          \DMO{\finmod}{FinMod}
          \DMO{\leaves}{leaves}
          \DMO{\planar}{planar}
          \DMO{\vertices}{vertices}

\nc{\norm}[2]{{ \ensuremath{\|} #1 \ensuremath{\|}}_{#2}}
\nc{\Dbar}[1]{\ensuremath{{\bar{\partial}}_{#1}}}
\nc{\Ce}{\ensuremath{\mathbb{C}}}
\nc{\B}{\ensuremath{\mathbb{B}}}
\nc{\osc}{\operatorname{osc}}
\nc{\leng}{\operatorname{leng}}

          \nc{\lagg}{\lag^{\cG}}
          \nc{\iso}{\mathsf{Iso}}
          \nc{\Set}{\mathsf{Set}}
          \nc{\Ass}{\mathsf{ \bf Ass}}
          \nc{\Mod}{\mathsf{Mod}}
          \nc{\modules}{\mathsf{Mod}}
          \nc{\sset}{\mathsf{sSet}}
          \nc{\liou}{\mathsf{Liou}}
          \nc{\poset}{\mathsf{Poset}}
          \nc{\trno}{T^*\RR^n_{\geq 0}}
          \nc{\spectra}{\mathsf{Spectra}}
          \nc{\tensorfin}{\tensor^{\fin}}
          \nc{\lagptg}{\lag_{pt,pt}^{\cG}}
          \nc{\Fin}{\mathcal{F}\mathsf{in}}
          \nc{\lagnl}{\lag_{N,\Lambda}}
          \nc{\lagmlg}{\lag_{M,\Lambda}^{\cG}}
          \nc{\lagsplit}{\lag^{\mathsf{split}}}
          \nc{\lagktimes}{(\lag^{\dd k})^\times}
          \nc{\lagplanar}{\lag^{\times,\planar}}
          \nc{\Cont}{\text{\rm Cont}}
          \nc{\Ham}{\text{\rm Ham}}
          \nc{\Dev}{\text{\rm Dev}}
          \nc{\Lin}{\text{\rm Lin}}
          \nc{\Int}{\text{\rm Int}}
          \nc{\Hom}{\text{\rm Hom}}
          \nc{\Chord}{\text{\rm Chord}}
          \nc{\nbhd}{\mathcal{N}\text{\rm{bhd}}}

          \nc{\onef}{1_{\fukaya}}
          \nc{\smsh}{\wedge}
          \nc{\un}{\underline}
          \nc{\xto}{\xrightarrow}
          \nc{\xra}{\xto}
          \nc{\tensor}{\otimes}
          \nc{\del}{\partial}
          \nc{\dd}{\diamond}
          \nc{\tri}{\triangle}
          \nc{\bb}{\Box}
          \nc{\into}{\hookrightarrow}
          \nc{\onto}{\twoheadrightarrow}
          \nc{\contains}{\supset}
          \nc{\transverse}{\pitchfork}
          \nc{\uncirc}{\underline{\circ}}
          \nc{\thetacontact}{\theta} 

          \nc{\Jbar}{\overline{J}}
          \nc{\Fbar}{\overline{F}}
          \nc{\delbar}{\overline{\del}}
          \nc{\thetabar}{\overline{\theta}}
          \nc{\omegabar}{\overline{\omega}}
          \nc{\Liou}{\text{\rm Liou}}

          \nc{\Yhat}{\widehat{Y}}

          \nc{\Mliou}{M}

          \nc{\vece}{ {\vec \epsilon}}	
          \nc{\vecd}{ {\vec \delta}}
          \nc{\ov}{\overline}
          \DMO{\op}{op}
          \nc{\opp}{ ^{\op}}

          \nc{\hiro}{\textcolor{blue}}
          \nc{\YG}{\textcolor{orange}}
		  \nc{\eqn}{\begin{equation}}
          \nc{\eqnn}{\begin{equation}\nonumber}
          \nc{\eqnd}{\end{equation}}
          \nc{\enum}{\begin{enumerate}}
          \nc{\enumd}{\end{enumerate}}
          \nc{\beastar}{\begin{eqnarray*}}
          \nc{\eeastar}{\end{eqnarray*}}

\numberwithin{equation}{section}

\def\R{{\mathbb R}}
\def\osc{{\hbox{\rm osc }}}

\def\E{{\mathbb E}}
\def\Z{{\mathbb Z}}
\def\C{{\mathbb C}}
\def\R{{\mathbb R}}

\def\N{{\mathbb N}}

\def\11{{\mathbb I}}

\def\Jbar{{\widetilde J}}

\def\delbar{{\overline \partial}}

          \def\cG{\mathcal G}

          \def\cW{\mathcal W}\def\c\Mliou{\mathcal \Mliou}

          \def\RR{\mathbb R}
          \def\\Mliou\Mliou{\mathbb \Mliou}

          \def\s\Mliou{\mathsf \Mliou}

          \def\b\Mliou{\mathbf \Mliou}

          \def\f\Mliou{\mathfrak \Mliou}

\def\Jbar{{\widetilde J}}

\def\delbar{{\overline \partial}}

\def\b{\beta}
\def\c{\chi}

\def\f{\phi}

\def\s{\sigma}

\def\CB{{\mathcal B}}

\def\CG{{\mathcal G}}
\def\CH{{\mathcal H}}

\def\CJ{{\mathcal J}}

\def\CM{{\mathcal M}}

\def\CP{{\mathcal P}}

\def\CP{{\mathcal P}}

\def\CT{{\mathcal T}}
\def\CU{{\mathcal U}}
\def\CV{{\mathcal V}}

\def\darr#1{\raise1.5ex\hbox{$\leftrightarrow$}
\mkern-16.5mu #1}

\def\roughly#1{\raise.3ex\hbox{$#1$\kern-.75em
\lower1ex\hbox{$\sim$}}}

\def\opname#1{\mathop{\kern0pt{\rm #1}}\nolimits}

\def\End{\opname{End}}

\def\vol{\opname{vol}}

\def\dist{\opname{dist}}

\def\supp{\operatorname{supp}}

\def\Dev{\operatorname{Dev}}

\def\leng{\operatorname{leng}}

\def\End{\operatorname{End}}

\def\Aut{\operatorname{Aut}}
\def\coker{\operatorname{Coker}}

\def\Cont{\operatorname{Cont}}

\def\Sing{\operatorname{Sing}}

\def\Image{\operatorname{Image}}
\def\ev{\operatorname{ev}}
\def\Int{\operatorname{Int}}

\def\ben{\begin{enumerate}}
\def\een{\end{enumerate}}
\def\be{\begin{equation}}
\def\ee{\end{equation}}
\def\bea{\begin{eqnarray}}
\def\eea{\end{eqnarray}}
\def\beastar{\begin{eqnarray*}}
\def\eeastar{\end{eqnarray*}}
\def\bc{\begin{center}}
\def\ec{\end{center}}

\renewcommand{\b}{\beta}

\def\Hoch{{\tt Hoch}}

\def\Cont{\operatorname{Cont}}

\def\Sing{\operatorname{Sing}}

\def\Ham{\operatorname{Ham}}

\def\Graph{\operatorname{Graph}}
\def\id{\operatorname{Id}}

\def\Riem{\operatorname{Riem}}

\def\E{\ifmmode{\mathbb E}\else{$\mathbb E$}\fi} 
\def\N{\ifmmode{\mathbb N}\else{$\mathbb N$}\fi} 
\def\R{\ifmmode{\mathbb R}\else{$\mathbb R$}\fi} 
\def\Q{\ifmmode{\mathbb Q}\else{$\mathbb Q$}\fi} 
\def\C{\ifmmode{\mathbb C}\else{$\mathbb C$}\fi} 
\def\H{\ifmmode{\mathbb H}\else{$\mathbb H$}\fi} 
\def\Z{\ifmmode{\mathbb Z}\else{$\mathbb Z$}\fi} 

\def\Hoch{{\tt Hoch}}

\def\Cont{\operatorname{Cont}}

\def\Sing{\operatorname{Sing}}

\def\Ham{\operatorname{Ham}}

\def\Graph{\operatorname{Graph}}

\def\darr#1{\raise1.5ex\hbox{$\leftrightarrow$}
\mkern-16.5mu #1}

\def\roughly#1{\raise.3ex\hbox{$#1$\kern-.75em
\lower1ex\hbox{$\sim$}}}

\def\opname#1{\mathop{\kern0pt{\rm #1}}\nolimits}

\def\End{\opname{End}}

\def\vol{\opname{vol}}

\def\dist{\opname{dist}}

\def\supp{\operatorname{supp}}

\def\Dev{\operatorname{Dev}}

\def\leng{\operatorname{leng}}

\def\End{\operatorname{End}}

\def\Aut{\operatorname{Aut}}
\def\coker{\operatorname{Coker}}

\def\Cont{\operatorname{Cont}}

\def\Sing{\operatorname{Sing}}

\def\Image{\operatorname{Image}}
\def\ev{\operatorname{ev}}
\def\Int{\operatorname{Int}}

\begin{document}

\quad \vskip1.375truein

\title[Injectivity radius lower bound]{Injectivity radius lower bound of convex sum of 
tame Riemannian metrics and applications to symplectic topology}

\author{Jaeyoung Choi and Yong-Geun Oh}
\address{POSTECH, 77 Cheongam-ro, Nam-gu, Pohang-si, Gyeongsangbuk-do, Korea 37673 \&
 Center for Geometry and Physics Institute, for Basic Science (POSTECH Campus)
79, Jigok-ro 127beon-gil, Nam-gu, Pohang-si, Gyeongsangbuk-do,  Korea 37673}
\email{jaeyoungkun@postech.ac.kr}
\address{Center for Geometry and Physics Institute, for Basic Science (POSTECH Campus)
79, Jigok-ro 127beon-gil, Nam-gu, Pohang-si, Gyeongsangbuk-do,  Korea 37673, \&
POSTECH, 77 Cheongam-ro, Nam-gu, Pohang-si, Gyeongsangbuk-do, Korea 37673}
\email{yongoh1@postech.ac.kr}

\begin{abstract} 
Motivated by the aspect of large-scale symplectic topology, we prove that
for any pair $g_0, \, g_1$ of smooth complete Riemannian metrics of bounded curvature and \emph{of injectivity
radius bounded away from zero}, the convex sum $g_s: = (1-s ) g_0 + s g_1$ also has
bounded curvature depending only on the curvature bounds $\|R_{g_i}\|_{C^0}$
of $g_0$ or $g_1$, and that  the injectivity radii of $g_s$ have uniform lower bound
depending only on the derivative bounds 
$\|R_{g_i}\|_{C^1} = \|R_{g_i}\|_{C^0} + \|DR_{g_i}\|_{C^0}$. 
 A main technical ingredient
to establish the injectivity radius lower bound is an application of the 
\emph{quantitative inverse function theorem}. Using these estimates, we prove
that each \emph{quasi-isometry} class of tame metrics is convex \emph{for all finite regularity class of
$3 \leq r < \infty$.} Using this Riemannian geometry result, we prove that
the set of smooth \emph{$C^3$-tame} almost complex structures inside the same quasi-isometry class
associated to the symplectic form $\omega$ is contractible in any $C^k$ topology for all finite $k \geq 0$.
\end{abstract}

\keywords{
quasi-isometry, convex sum of metrics, geodesic flow, injectivity radii lower bounds,
quantitative inverse function theorem,  strong $C^r$ topology,  $\mathfrak T$-tame almost complex structure, $C^3$-tameness of almost complex structure}

\subjclass[2020]{53C20, 53D35}
\thanks{This work is supported by the IBS project \# IBS-R003-D1.}
\date{}

\maketitle

\tableofcontents

\section{Introduction}

In a recent work of the present authors  \cite{choi-oh:construction},
we construct a Fukaya category on infinite-type surfaces and prove that
the $A_\infty$ category is not quasi-equivalent to the colimit of  Fukaya categories of
finite-type surfaces. This category is an invariant under the
deformation of tame $J$'s that is compatible to a \emph{fixed} Riemannian metric of bounded
curvature but not when the Riemannian metric goes out of the quasi-isometry class of the given
metric.  In this regard, the invariants arising from \cite{choi-oh:construction} (or any symplectic
invariants relying on the structure of ideal boundary on noncompact symplectic manifolds in that matter) are
not exactly symplectic invariants but are \emph{large-scale geometric symplectic invariants!}
This case concerns noncompact symplectic manifolds of infinite type.

In general higher dimensional situations, if a noncompact symplectic manifold $M$ has a suitably good
compactification $\widehat M$ such as a complete Liouville manifold so that 
 $\widehat M$ is given by the union
$$
M \sqcup \del_\infty M
$$
together with the Liouville embedding $\del_\infty M \times [0, \infty) \hookrightarrow M$,
then the Liouville symplectic form $\omega = d\lambda$ carries a canonically associated
quasi-isometry class $\mathfrak T$ induced by the image of
$\del_\infty M \times [0,\infty)$ in $M$. (See Remark \ref{rem:symplectization} for the description of
such a quasi-isometry class.)  The growth rate invariants of symplectic homology of affine algebraic varieties
is another such large-scale symplectic topological invariants.
(See \cite{seidel:biased}, \cite{mclean:growthrate} for the study thereof.)
This case concerns noncompact symplectic manifolds of finite type.

\emph{All objects of consideration of the present paper, such as Riemannian metrics and almost complex structures
are assumed to be of $C^\infty$-class unless otherwise said.}

\subsection{Statements of main results: symplectic topology}

Some more background and motivation of our study of large scale Riemannian geometry
presented in the present article are now in order.
The well-known Gromov's lemma \cite{gromov:invent} on noncompact
 symplectic manifolds  holds in weak $C^\infty$ topology but
 fails to hold in strong $C^\infty$ topology if \emph{two metrics associated to the almost complex structures
 tame to a given symplectic form $\omega$ are allowed to vary beyond 
 their quasi-isometry class},
or not necessarily bilipschitz equivalent. (See Example 
\ref{exam:2-dim} for such an example.)
 It is generally said that an almost complex structure 
$J$ on a symplectic manifold $(M,\omega)$ is called $\omega$-tame if the bilinear form
$g_J: =\omega(\cdot, J\cdot)$ is a symmetric positive definite, i.e., defines a Riemannian metric
and there exists a tame metric $g$ that satisfies \eqref{eq:A-quasiisometry}
for some $A \geq 1$. Denote by $\CJ_\omega$ the set of such almost complex structures. 

However to do the geometric
analysis of pseudoholomorphic curves on noncompact symplectic manifolds, it is important to 
\emph{fix} a tame behavior of the metric $g$ appearing here because  the Riemannian
metrics $g_{J_0}$ and $g_{J_1}$ associated to two $\omega$-compatible
almost complex structures $J_0$ and $J_1$  are not necessarily quasi-isometric as mentioned above.  
When this happens, any symplectic invariants 
constructed using $J_0$ and $J_1$ via the pseudoholomorphic curves have no reason
to be the same, \emph{when the construction involve the ideal boundary of the noncompact
symplectic manifolds}. The Fukaya category constructed for the infinite-type surface in 
\cite{choi-oh:construction} or the aforementioned growth rate of symplectic homology of 
affine algebraic varieties are examples of such invariants. A difference between the two 
situations is that the former concerns the case of non-cylindrical end while the latter
does the case of cylindrical end.

The results on Riemannian geometry in the present paper
and their implications to symplectic topology indicate that 
for the applications of the methodology of pseudoholomorphic curves to 
large scale symplectic topology of a noncompact symplectic manifold $(M,\omega)$,
\emph{one needs to  introduce the following notion of $(\omega,\mathfrak T)$-tame
almost complex structures, and assume at least $C^3$-tameness
for the almost complex structures $J$ with respect to the symplectic form $\omega$.}
(See Definition \ref{defn:Ck-tame} below for the definition of $C^r$-tameness.)

\begin{defn}[$(\omega,\mathfrak{T})$-tame almost complex structures]
\label{defn:omegaT-tame-intro}
Let $\mathfrak{T}$ be a given quasi-isometry class of Riemannian metrics on $M$.
 We call $J$ a $C^r$ $(\omega,\mathfrak{T})$-tame almost complex structure  if the following hold:
 \begin{enumerate}
 \item It is $\omega$-tame.
 \item The Riemannian metric $g_J = \omega(\cdot, J \cdot)$ is in 
 the quasi-isometry class $\mathfrak{T}$ and $C^r$-tame.
 \end{enumerate}
 We denote by $\CJ_{\omega;\mathfrak{T}}$ the set of such almost complex structures.
 \end{defn}

\begin{rem}\label{rem:strict-convexity}
To the best knowledge of the present authors, such a $C^3$-requirement has not been
recognized in the previous literature of symplectic topology. We believe that this $C^3$-tameness 
should be mentioned to ensure contractibiliy of $\omega$-tame 
almost complex structures \emph{with positive injectivity radius},
and required in the definition of tame almost complex structures 
 for the purpose of constructing (large scale) symplectic topological
invariants of noncompact symplectic manifolds.  This is because the natural candidate of
the contraction to a reference tame metric $g_{\text{\rm ref}}$ given by
\be\label{eq:contraction}
(t,g)  \mapsto t g_{\text{\rm ref}} + (1-t) g
\ee
may not have a positive injectivity radius lower bound over $t \in [0,1]$ with $C^r$-tameness
with $r < 3$. See the next subsection, especially Remark \ref{rem:non-convexity},
for further expounding of this Riemannian geometric difficulty.
\end{rem}

Represent the given $\mathfrak{T}$ by a Riemannian metric $g$.
A priori the answer to the question whether or not 
the subset $\CJ_{\omega;\mathfrak{T}} \subset \CJ_{\omega}$
is connected is not known because convexity of the set
\be\label{eq:gJ}
\{g_J \mid J \in \CJ_{\omega;\mathfrak{T}} \} =: \Riem_{\omega;\mathfrak{T}}(M)
\ee
is not known  as indicated in
Remark \ref{rem:strict-convexity} and \cite{nabutovsky}.  Recall that the 
`known proof' of contractibility relies on the 
presumption that `the set of tame metrics is strictly convex' so that the image of the
path \eqref{eq:contraction} is contained in \eqref{eq:gJ}.  

In this regard we prove that the following result,
 the precise statement of which we refer to Theorem \ref{thm:omegaTJ-contractible}.

\begin{theorem} \label{thm:omegaTJ-contractible-intro} Let $(M,\omega)$ be a tame
symplectic manifold. The set $\CJ_{\omega;\mathfrak{T}}$ consisting of \emph{smooth} $C^r$-tame $r \geq 3$ is
contractible with respect to the strong $C^k$ topology for $0 \leq k < \infty$.
\end{theorem}
We refer readers to \cite{hirsch} for the precise definition of strong $C^r$ topology which we
briefly recall in the beginning of Section \ref{sec:contractibility-g}.

An immediate corollary of Theorem \ref{thm:omegaTJ-contractible-intro}
 is the following special case of two-dimensions. This 
is needed in the verification that the aforementioned Fukaya category constructed in 
\cite{choi-oh:construction} is indeed a large-scale symplectic topological invariant
depending only on the quasi-isometry class of tame Riemannian metric.
(See Theorem \ref{thm:convexity} stated below to see how the aspect of 
this Riemannian geometry enters.)

\begin{cor}[Theorem \ref{thm:surface-case}]
Let $\Sigma$ be a noncompact surface equipped with hyperbolic structure.
Denote by $\mathfrak T$ a quasi-isometry class of hyperbolic structures of $\Sigma$ and by $\Riem_\mathfrak T(\Sigma)$
the set of Riemannian metrics quasi-isometric to the given hyperbolic metric. Then
the set $\Riem_\mathfrak T(\Sigma)$ consisting of $C^r$-tame metrics of $3 \leq r < \infty$
is contractible in strong  $C^k$ topology for finite $0 \leq k < \infty$.
\end{cor}

The finding of the present article is a byproduct of the present authors'
work \cite{choi-oh:construction}  on the Fukaya category of infinite-type
surfaces in which their category strongly relies on the underlying
hyperbolic structure of Riemann surfaces. 

\subsection{Statements of main results: Riemannian geometry}

The standard proof of contractibility of the set of almost complex structures tame to a given symplectic form 
is based on the contractibility of relevant tame Riemannian metrics. The proof of the 
contractibility of tame Riemannian metrics is based on the convexity thereof.
\begin{rem}\label{rem:non-convexity}
 The main source of the requirement $r \geq 3$ for this $C^r$ tameness 
lies in the fact that \emph{the injectivity radius function 
$g \mapsto \iota_g$  is not a continuous function even with the strong $C^\infty$ topology on
a noncompact manifold.} Because of this,  even when two complete Riemannian metrics
$g_0$ and $g_1$ have positive injectivity radii $\iota_{g_1}, \, \iota_{g_2}> 0$,  it is not a trivial matter
to check in general, even with the $C^\infty$ topology,
whether the positivity
$$
\inf_{t \in [0,1]}\{\iota_{g_t} \mid g_t: = (1-t) g_0 + t g_1\} > 0
$$
holds or not for the convex sum $t \mapsto (1-t) g_0 + tg_1$. (See \cite{nabutovsky} for the reason why.)
The finding of the present paper shows that we
need at least $C^3$-tameness for this positivity. At the moment it is not known whether
this infimum is positive or not when $k \leq 2$. 
\end{rem} 

To state the main results of the present paper on the global Riemannian geometry, we need to borrow the standard definitions of
quasi-isometry and bilipschitz equivalence from \emph{large-scale geometry}
or \emph{coarse geometry}.

\begin{defn}[Quasi-isometry]\label{defn:quasi-isometry} Two Riemannian metrics
$g_1, \, g_2$ on a smooth manifold are said to be quasi-isometric if there exists 
a constant $A\geq 1$ such that
\be\label{eq:A-quasiisometry}
\frac1A\, g_1(u,u)  \leq  g_2(u,u) \leq A\, g_1(u,u)
\ee
for all $u\in T_xM$ for all $x \in M$.
\end{defn}

We also consider the following equivalence relation, which is a macroscopic version of
the above quasi-isometry of Riemannian metrics. 

\begin{defn}[Bilipschitz]\label{defn:A-bilipschitz}
Two Riemannian metrics
$g_1, \, g_2$ on a smooth manifold are said to be \emph{$A$-bilipschitz}
if their associated distance functions are $A$-bilipschitz,
i.e., if there exists a constant $A\geq 1$ such that
\be\label{eq:A-bilipschiz}
\frac1A d_{g_1}(x,y)   \leq  d_{g_2}(x,y) \leq A d_{g_1}(x,y)
\ee
for all $x,\, y \in M$. Here we denote by $d_{g_i}$ the distance function of $g_i$.
We just say $g_1, \, g_2$ are \emph{bilipschitz} if they are $A$-bilipschitz for some $ 1\leq A < \infty$.
We call the associated equivalence class  a \emph{bilipschitz class}.
\end{defn}
It is easy to see that for the Riemannian distance metrics the above two notions of
the above quasi-isometry and the bilipschitz equivalence are equivalent.
(However one should recall that comparing two notions of ($C^0$) \emph{quasi-isometric equivalence} and
\emph{bilipshitz} in general large-scale metric geometry or in coarse geometry is a highly nontrivial problem.)
 
We will denote by $\mathfrak T$ a quasi-isometry class or equivalently a bilipschitz class
of a Riemannian metric. In the present paper, we will interchangeably use both
terms as we feel more appropriate depending on the circumstances.
We will also use the following terminology for the simplicity of exposition.
We denote by $\inj_g: M \to \R_+$ the function of pointwise injectivity radius $x \mapsto \inj_g(x)$.

\begin{defn}\label{defn:Ck-tame}
 Suppose that a complete Riemannian metric $g$ has positive injectivity radius lower bound
$\iota_g: = \inf_{x \in M}\inj_g(x) > 0$.
We call a metric $g$ $C^r$-\emph{tame} if there exists  a constant
$C_\ell > 0$ for each $0 \leq \ell \leq r -2$ such that
$$
\|D^\ell R_g\|_{C^0}=\sup_{x \in M} |D^\ell R_g(x)| < C_\ell.
$$
We just say that $g$ is tame if this holds for all $r \geq 2$.
\end{defn}	
From now on, we will always denote by $\mathfrak{T}$ a quasi-isometric class of
tame Riemannian metrics, and by $\Riem_{\mathfrak T}(M)$
 the set of tame metrics in class $\mathfrak{T}$, unless explicitly mentioned otherwise.

The first main result of Part I of the present paper is  the following convexity
result of $\mathfrak{T}$.

\begin{theorem}[Theorem \ref{thm:Rs-bounds}, \ref{thm:DRs-bounds} \& \ref{thm:iota-bounds}]
\label{thm:convexity}
Let $g_0, \, g_1$ be a pair of complete $C^3$-tame Riemannian metrics in the same
quasi-isometry class. Suppose $\|R_{g_k}\|_{C^1} \leq C$ 
and $\iota_{g_k} > 0$ for $k = 0, \, 1$. 
Consider the convex sum $g_s = (1-s)g_0 + sg_1$, and 
put $\epsilon_k =\iota_{g_k} > 0$ for $k = 0, \, 1$. Then the following hold: 
\begin{enumerate} 
\item There exists $\mathfrak C = \mathfrak C(C,\epsilon_0,\epsilon_1) > 0$ such that 
\be
\|R_s\|_{C^1} \leq \mathfrak C.
\ee
\item There exists some $\epsilon' = \epsilon'(C, \epsilon_0, \epsilon_1) > 0$
such that
$$
\iota_{g_s} \geq \epsilon' > 0
$$
for all $s \in [0,1]$.
\end{enumerate}
In particular, the set $\Riem_{\mathfrak T}(M)$ is strictly convex.
\end{theorem}

We would like to recall readers that estimating such a lower bound of the injectivity
radius 
is a hard work as demonstrated by \cite{nabutovsky}, and involves certain volume control
\cite{cheeger:finiteness}, \cite{CGT} in general. An upshot of the theorem is that the
statement does not involve any volume control but that it crucially relies on 
the uniform bound for the $C^2$-norm of the exponential maps:  This is one of the reasons 
why the derivative bound of the curvature, i.e., the $C^3$-tameness enters.

An outcome of these estimates, together with uniform curvature bound which will be also
proved, enables us to prove that \emph{provided $3 \leq r < \infty$},
\begin{enumerate}
\item the  convex sum $s \mapsto (1-s) g_0 + sg_1$ is contained in $\Riem_{\mathfrak{T}}(M)$,
\item it also defines a continuous path therein in strong $C^r$ topology. 
\end{enumerate}
Such a convexity of $\Riem_{\mathfrak{T}}(M)$ 
fails to hold for the  strong $C^\infty$ topology in general, or under the lower regularity 
 $1 \leq r < 3$. The latter is  mainly because of the failure of Theorem \ref{thm:convexity}
 with that low regularity.
 
\begin{theorem}[Contractibility in strong \emph{$C^k$} topology]\label{thm:contractibility-intro} 
Let $\mathfrak T$ be any quasi-isometry class 
of $C^r$-tame Riemannian metrics on $M$ with $3 \leq r < \infty$. Then 
  $\Riem_{\mathfrak{T}}(M)$ is contractible in strong $C^k$ topology for all $0 \leq k < \infty$.
\end{theorem}

In this regard, we do not know the answer to the following question, mainly because we do not
have the uniform injectivity radii lower bound for the path \emph{under $C^2$ tameness.}

\begin{ques} Is the set
$\Riem_{\mathfrak T}(M)$ strictly convex for a noncompact Riemannian manifold $M$, or 
is the path $s \mapsto (1-s) g_0 + sg_1$ continuous in
the subspace topology of $\Riem_{\mathfrak T}(M) $ in the strong $C^2$-topology 
of $ \Riem(M)$?
\end{ques}
We refer to Appendix \ref{sec:direct-limit} for some discussion on the
contractibility with respect to some $C^\infty$ topology, which we call
the direct limit $C^\infty$ topology.

\subsection{Outline of the proofs}

For the proof of uniform curvature bounds, we utilize the formula of the curvature operator
$R_{g_s}$ of $g_s$ of \emph{convex sum} $g_s: = (1-s) g_0 + s g_1$
that is obtained in a recent article by Cavenaghi and Speranca
\cite[Proposition 2.1]{cavenaghi-speranca}.

The main part of the  proof lies in that of the injectivity lower bound. 
The injectivity radii can jump under a continuous deformation even under the curvature bound
because of the appearance or disappearance of a short geodesic which reflects a nonlocal
behavior of metrics.  By now it is well-known that uniformly controlling the injectivity radii
even on compact manifolds is a difficult task in general. (We refer readers to \cite{nabutovsky} to see
why this is so.) Recall that a lower bound for the injectivity radius  can be estimated
by estimating the infimum of the radii of geodesic normal  balls over the points of $M$.
In particular we need to study the \emph{injectivity of the exponential maps} to estimate the radii 
of such balls. 

The main novelty of the present work is that the estimate of such radii  
under a continuation of metrics  can be nicely done, although the estimate is not explicit, by exploiting 
the aforementioned bounds of the curvature and its derivative, and some idea of the proof of a 
\emph{quantitative inverse function theorem}. For this purpose, we need the exponential maps
 to have uniform $C^2$-bounds.  This is where our hypothesis on $C^3$-tameness of $g_0$ and $g_1$
 enters. We would like to highlight that  the proofs of the quantitative inverse function theorem, 
 for example, from  \cite[Corollary 2.5.6]{abraham-marsden-ratiu}, \cite[Section 8]{christ:hilbert} 
 then enable us to rule out the appearance of short closed geodesics which is the main obstacle to have
 positive injectivity of radius bound, in the presence of a uniform derivative bound of the curvature. 
 
To help readers get the overall scheme of our proof, we explain how each given condition put in the statement of
Theorem \ref{thm:iota-bounds} is used in the proof:
\begin{enumerate}
\item Injectivity radius bound of $g_0$ provides an atlas of $M$ with a uniform 
size of the coordinate geodesic normal balls of $g_0$.
\item Curvature bound $\|R_{g_0}\|_{C^0}$ and $\|R_{g_1}\|_{C^0}$ provides 
a uniform bound of $\|R_{g_s}\|_{C^0}$.
\item Quasi-isometry hypothesis on $g_0$ and $g_1$ and completeness thereof imply that 
$g_s$ are also complete and quasi-isometric thereto.
\item Completeness of $g_s$ and the bound for the derivatives $DR_{g_0}$ and $DR_{g_1}$ 
first implies the existence of a common domain  and  a codomain 
of the maps $F_i$ (see  \eqref{eq:Fi}), and then imply the uniform $C^2$ bound for the maps $F_i$.
\item Combining the above all, we can apply the quantitative inversion function 
theorem (Theorem \ref{thm:IFT}) to conclude the uniform positive injectivity lower bound.
\item For the clarity of exposition, we consistently use the letter `r'  to talk about $C^r$-tameness,
while we use the letter `k'  to talk about $C^k$-topology, respectively.
\end{enumerate}

Once these basic estimates, especially the injectivity radius lower bound, are obtained,
the proof of Theorem \ref{thm:Cr-contractibility} will follow from the definition of strong $C^k$
topology of $\Riem_{\mathfrak T}(M)$, which is nothing but the subspace topology of the strong
 $C^k$ topology of $\Riem(M)$.

An interesting byproduct of this scheme of the proof is that any (finite-time)
$C^3$-continuation of a tame metric on open manifold cannot develop a cusp,
which seems to carry some interest of its own. Such a non-collapsing result 
\emph{under the uniform curvature bound} in the study of the Ricci flow plays 
an important role for the application of Ricci flow to the 3-dimensional 
topology. (See \cite{perelman:finite-time-extinction}, \cite[Chapter 8]{morgan-tian}.)
In relation to the collapsing phenomenon under the Ricci flow, a finite-time extinction
may be rephrased as a divergence in strong $C^\infty$ topology of the Riemannian 
metric under the Ricci flow in finite time.

\begin{theorem}\label{thm:noncollapsing-intro} Consider any continuous family $\{g_s\}_{s \in [0,1]}$
in strong $C^3$ topology of complete Riemannian metrics with
$$
\|R_{g_0}\|, \, \|DR_{g_0}\|_{C^0} < C, \quad \iota_{g_0} > 0.
$$
Then there is a constant $C' = C'(C, \{g_s\}), \, \epsilon' = \epsilon'(C, \epsilon, \{g_s\}) > 0$
with $\epsilon: =  \iota_{g_0}$ such that
\be
\inf_{s \in [0,1]} \iota_{g_s} > \epsilon'
\ee
\end{theorem}
Recalling that as the thin cylinder example \cite{CGT}
or the cuspidal  hyperbolic Riemann surfaces shows, the curvature bound itself
does not provide the injectivity radius lower bound for a general \emph{single individual}
metric. The upshot of this theorem is that  such a finite-time collapsing, or rather forming a cusp,
 cannot arise under a deformation of metrics that is continuous in strong $C^3$ topology.

\medskip

\noindent{\bf Acknowledgement:} 
We thank Gang Tian
and Bruce Kleiner for useful email communications on the large-scale geometry,
and the unknown referee for useful comments which helps us to improve the presentation 
of the paper.

\bigskip

\noindent{\bf Notations:}
\begin{enumerate}
\item $\vec a$; a vector in $\R^n$,
\item $v$; an element in the tangent space $T_pM$,
\item $B^n(r)$; the standard open ball of radius $r$ centered at the origin of $\R^n$,
\item $B_r^g(p)$; the geodesic normal open ball of radius $r$ centered at $p \in M$,
\item $I_p^g: \R^n \to T_pM$; the canonical isometry with respect to a given orthonormal frame
$\mathscr B$ of the inner product space $(T_pM, g_p)$.
\end{enumerate}

\part{Global Riemannian geometry}

\section{Convex sum of complete Riemannian metrics}

Assume that both $g_0$ and $g_1$ are complete and of bounded geometry
in the present section, and consider the convex sum thereof
$$
g_s = (1-s) g_0 + sg_1, \quad s \in [0,1].
$$
We will show that for any given reference tame metric $g_{\text{\rm rf}}$, the map
$$
(s,g) \mapsto (1-s) g + s g_{\text{\rm rf}}
$$
defines a contraction to the point $g_{\text{\rm rf}}$ that is continuous in strong $C^k$
topology of the set of $C^r$-tame metrics for $r \geq 3$.
(It is not a priori contractible in the usual definition of strong $C^\infty$ topology.
See Remark \ref{rem:direct-limit-topology} for relevant comments.)

The main steps for the proof of Theorem \ref{thm:convexity} 
are the estimates of the curvatures and the
injectivity radii bounds, especially the latter,
 for this convex combinations. Recently Cavenaghi and
Speranca \cite{cavenaghi-speranca}
 studied the curvature property of this convex sum
in terms of the given metrics $g_0, \, g_1$ for a different purpose.
We use their explicit curvature formula to obtain some explicit bound for the curvature
operator of $g_s$ in terms of those of $g_0$ and $g_1$.

We start with  the following proposition on the completeness.

\begin{prop} \label{prop:A-bllipschitz-convex} Suppose
$$
d_0, \, d_1: \CM \times \CM \to \R_{\geq 0}
$$
are  two complete metrics on the same topological space $\CM$.
 If $d_0, \, d_1$ are $A$-bilipschitz with $A \geq 1$ in the sense of
Definition \ref{defn:A-bilipschitz},  $d_s$ is again $A$-bilipschitz with $d_0, \, d_1$.
In particular the space of $A$-bilipschitz complete metrics on $\CM$ is convex.
\end{prop}
\begin{proof} We first prove $A$-bilipschitz property of $d_s$.
A direct calculation shows that $d_s$ satisfies
$$
\left((1-s) +\frac{s}{A}\right) d_0(x,y) \leq d_s(x,y) \leq ((1-s) + sA) d_0(x,y)
$$
and
$$
\left(\frac{1-s}{A} +s \right) d_1(x,y) \leq d_s(x,y) \leq ((1-s)A + s) d_1(x,y).
$$
Since $A \geq 1$, these inequalities show $d_s$ is $A$-bilipschitz with $d_0, \, d_1$.

Completeness then immediately follows since the property is preserved under
the bilipschitz equivalence.
\end{proof}

To prove the convexity of the bilipschitz class of tame metrics, we
consider the convex sum of $g_s = (1-s) g_0 + s g_1$ of
$g_0, \, g_1$. We have only to prove that $g_s$ is of bounded geometry, i.e.,
\begin{enumerate}
\item there exists a constant $C> 0$ such that $\|R_{g_s}\|_{C^0} < C$
for the sectional curvature $R_{g_s}$.
\item there exists $\epsilon > 0$ such that $\iota_{g_s} > \epsilon$ for all
$s \in [0,1]$.
\end{enumerate}

For the later purpose, we also consider the following $C^0$-distance between two
metrics as a section of the bundle of symmetric positive definite quadratic forms.

\begin{defn}[Quasi-isometric ratio] Let $g_0, \, g_1$ be two Riemannian metrics of $M$. We define
\beastar
A^+_x(g_0,g_1) &: = & \sup \{ |v|_{g_1} \mid v \in T_x M, \, |v|_{g_0} = 1\}\\
A^-_x(g_0,g_1) &: = & \inf \{ |v|_{g_1} \mid v \in T_x M, \, |v|_{g_0} = 1\}
\eeastar
and
\be\label{eq:Apm}
A(g_0,g_1)  : =  \sup_{x \in M} \max\{A^+_x(g_0,g_1), \, 1/A^-_x(g_0,g_1)\}
\ee
We call $A(g_0,g_1)$ the \emph{quasi-isometric ratio} of $g_0, \,g_1$.
\end{defn}
By the definition of $A(g_0,g_1)$, $A(g_0,g_1)$ is finite if and only if $g_0$ and $g_1$ are quasi-isometric.
It also follows that $A(g_0,g_1)=A(g_1,g_0) \geq 1$, and that $A(g_0,g_1) = 1$ if and only if $g_0$ and $g_1$
are isometric. 

\begin{rem} Consider the following standard notion in the comparison geometry:
For a given pair of Riemannian metrics $g_0, \, g_1$, we define the function 
$M_{g_0,g_1}: M \to \R_+$ given by
$$
M_{g_0,g_1}(x) = \sup_{0 \neq v \in T_x M} \left|\log\left(\frac{|v|_{g_1}}{|v|_{g_0}}\right)\right|.
$$
(See Definition \ref{defn:quasiisometric-ratio}.)
It follows that $M_{g_0,g_1}$ is a continuous function on $M$. Then we have
$$
\log A(g_0,g_1) = \sup_{x \in M} M_{g_0,g_1}(x).
$$
\end{rem}

\section{Bounds for the sectional curvature of the convex sum}
\label{sec:Rs-bound}

We first consider the curvature estimate. We would like to recall readers that
the curvature \emph{quadratically} depends on the metric and its derivatives
up to second order, and that considering the associated metric on the cotangent bundle 
amounts to taking the inverse of the metric coefficients and their derivatives. For our purpose,
we will also need the bound for the derivative of the curvature and so will need 
\emph{the bound for the derivatives of the metric up to the third order}.

This is the reason why getting  the uniform bound for the curvature and its derivative
of the \emph{convex sum} $g_s$ in terms of the bounds for those of $g_0, \, g_1$
needs to be verified, especially to obtain some explicit bound
depending on the curvatures (and other tensorial expressions) and its derivatives of $g_0, \, g_1$.

For this purpose, we follow the strategy used by Cavenaghi and Speranca 
\cite{cavenaghi-speranca} in the following discussion. Consider the endomorphism
$P = P_{g_0;g_1}\in \End(TM)$ determined by
$$
g_0(PX,Y) = g_1(X,Y)
$$
which is positive definite symmetric with respect to $g_0$. We also consider
another $ D = D_{g_0,g_1} \in \End(TM)$  given  by the difference
$$
D_{g_0,g_1} = \nabla^1 - \nabla^0
$$
where $\nabla^i$ are the Levi-Civita connections of $g_i$ for $i=0, \, 1$ respectively.
(Recall that the difference of two affine connections on $TM$ defines a $(1,1)$ tensor field.)

The following formula is proved  in \cite{cavenaghi-speranca}.

\begin{prop}[Compare with Proposition 2.1 \cite{cavenaghi-speranca}]
Let $R_s$ be the curvature
operator associated to $g_s$. Then
\bea\label{eq:Rs}
R_s(X,Y, Y,X) & = & (1-s) R_0(X,Y,Y,X) + s R_1(X,Y,Y,X)  \nonumber\\
& {} & + s(1-s) g_1\left((1-s) \text{\rm Id} + s P)^{-1} D(X,Y),D(Y,X)\right) \nonumber\\
&{}& -  s(1-s) g_1\left((1-s) \text{\rm Id} + s P)^{-1} D(Y,X), D (X,Y)\right).
\eea
\end{prop}
Here we mention that the endomorphism $(1-s) \text{\rm Id} + s P$ is
invertible since $P$ is positive definite (with respect to $g_0$).
We first state the following easy lemma.
\begin{lem} Let $\lambda_P > 0$ be the smallest eigenvalue of $P$. Then we have
$$
\left\|((1-s)\text{\rm Id} + s P)^{-1}\right\|_{C^0} \leq \frac1{1-s  + s \lambda_P}
\leq \frac1{\min\{1, \lambda_P\}}.
$$
\end{lem}

\emph{From now on, we will take $g_{\text{\rm rf}}$ as the given reference metric 
which is also tame, and measure all
relevant norms in terms of this metric $g_{\text{\rm rf}}$.}

Now suppose the curvature bounds
\be\label{eq:Ri-bounds}
\|R_{g_i}\|_{C^0} \leq C_i < \infty
\ee
for some constants $C_i > 0$ for $i=0, \, 1$.  Then
we can also fix the radii $r_i > 0$,  $i=0, \, 1$ such that the geodesic normal balls 
$B^{g_i}_{r_i}(p)$  are strongly convex at all points $p \in M$ for both metrics $g_i$ 
with $i = 0, \, 1$.

Since $g_0, \, g_1$ are bilipschitz and of bounded
curvature, we also have the following estimate for the coordinate change matrices
whose proof is essentially the same as that of \cite[Lemma 3.4]{cheeger:finiteness} and so
omitted.
\begin{lem}[Compare with Lemma 3.4 \cite{cheeger:finiteness}]\label{lem:cheeger} Let $g_0, \, g_1$ be as above.
Let $\{x^i\}$ and $\{y^i\}$ be geodesic normal coordinates of
$g_0$ and $g_1$ on $B_0 : = B_{r_0}^{g_0}(p)$ and $B_1: = B_{r_1}^{g_1}(p)$ respectively.
Given $S > 0$ with $\|R_{g_0}\|_{C^0}, \, \|R_{g_1}\| _{C^0} < S$, there exists a
constant $C = C(S)$ such that if $r_k < \frac{\pi}{2 \sqrt{S}}$ for $k=0,\,1$, then
\be\label{eq:coordinate-change}
\left|\frac{\del y^i}{\del x^j}(x)\right| < C(S)
\ee
for all $x \in B_0 \cap B_1$.
\end{lem}

An immediate corollary of the curvature bound and this lemma is the following
bound for the tensor $D$.
This is a consequence of the general principle that \emph{curvature bounds imply
bounds for the second derivatives of the metric tensor.}

\begin{prop} \label{prop:|D|} There
exists some constant $C' > 0$ depending only on $g_0, \, g_1$ (and $C_0, \, C_1$)
such that
\be\label{eq:Rs-bound}
\|D\|_{C^0}  \leq C'
\ee
\end{prop}
\begin{proof} The bounds for the curvature and the injectivity radius lower bounds imply that
we are given a pair of constants $r_0, \, r_1$ given in Lemma \ref{lem:cheeger}
such that we can choose a normal coordinate system on a convex geodesic balls  of the uniform size, 
say $\epsilon> 0$, at every point $p \in M$ simultaneously for $g_0, \, g_1$ so that
$$
B_{\epsilon}^{g_{\text{\rm rf}}}(p) \subset B_{r_0}^{g_0}(p) \cap B_{r_1}^{g_1}(p)
$$
and
$$
B_{\epsilon}^{g_i}(p) \subset B_{r_0}^{g_0}(p) \cap B_{r_1}^{g_1}(p).
$$
We fix such a constant $\epsilon > 0$.

We first consider $g_{\text {\rm rf}}$ and
denote the geodesic normal coordinates on $B_{\epsilon}^{g_{\text{\rm rf}}}(p)$ by
$(x^1, \cdots, x^n)$ at $p \in M$ of $g_0$
By definition of normal coordinates, the associated normal coordinates $(x^1,\cdots, x^n)$
satisfies
\be\label{eq:normality}
g_{ij}(p) = \delta_{ij}, \quad \Gamma^k_{ij}(p) = 0.
\ee
Then we have the following
Taylor expansion at $p$ of the metric tensor coefficients $g_{ij}$
\bea\label{eq:gij}
g_{ij}(x) & = & \delta_{ij} + \frac13 R_{ik\ell j}(p)x^kx^\ell + \frac16 R_{ik\ell j;s}(p) x^kx^\ell x^s
\nonumber\\
&& + \left(\frac1{20} R_{ik\ell j;st}(p) + \frac2{45} \sum_m R_{ik\ell m}(p)R_{jstm}(p)\right)
x^kx^\ell x^s x^t \nonumber\\
&& + O(r^5)
\eea
where $r$ is the distance from $p$. (See \cite[Equation (1.5)]{morgan-tian}
where the formula is attributed to Sakai \cite{sakai}.)

We recall the formula for the Christoffel symbol
\be\label{eq:Gammakij}
\Gamma^k_{ij} = \frac12 g^{k\ell}\left(\frac{\del g_{\ell j}}{\del x^i} + \frac{\del g_{i\ell}}{\del x^j}
- \frac{\del g_{ij}}{\del x^\ell}\right).
\ee
By a direct calculation of derivatives of the right hand side of \eqref{eq:gij} at the given point $p$
of $g_{ij}$ using the expression \eqref{eq:gij}
applied to $g_0$ and $g_1$ respectively;
$$
\frac{\del g_{ij}}{\del x^m}(x) = \frac13 \left(R_{imk j}(p)+ R_{ikmj}(p)\right) x^k + O(r^2)
$$
which implies
\be\label{eq:Gamma-norming0}
\|\Gamma^k_{ij}(g_0) \|_{C^0;g_0} \leq C_0 \|R_{g_0}\|_{C^0;g_0}
\ee
on $B_{r_0}^{g_0}(p)$. 
By the same token, we consider $g_1$ now and similarly have
\be\label{eq:Gamma-norming1}
\|\Gamma^k_{ij}(g_1) \|_{C^0;g_1} \leq C_1 \|R_{g_1}\|_{C^0; g_1}
\ee
on $B_{r_1}^{g_1}(p)$. Here the norms are measured by $g_0$ and $g_1$ respectively.

An immediate consequence of Lemma \ref{lem:cheeger} is that the bounds \eqref{eq:Gamma-norming0},
\eqref{eq:Gamma-norming1} can be converted to those in terms of the 
reference metric $g_{\text{\rm rf}}$,
$$
\|\Gamma^k_{ij}(g_k) \|_{C^0} \leq C_k' \|R_{g_k}\|_{C^0}, \quad k = 0, \, 1
$$
by adjusting $C_0, \, C_1$ slightly. We set $C' = \max\{C_0', C_1'\}$ where we have
$$
C' = C'(C_0, C_1, r_0, r_1).
$$
Then writing and substituting these estimates into the coordinate expression of
$D = \nabla^1 - \nabla^2$, we can find  the bound $C_k'$ depending only
on 
\begin{itemize}
\item  the curvature expression $R_{ik\ell j}^g(p)$ for $g = g_0, \, g_1$
 in coordinate functions $x^i$'s respectively, 
 \item the metric coefficients of $g_0, \, g_1$ and their inverses.
 \end{itemize}
These are valid on the geodesic ball 
$$
B_\epsilon^{g_{\text{\rm rf}}}(p) \subset B_{r_0}^{g_0}(p) 
\cap B_{r_1}^{g_1}(p).
$$
\end{proof}

Now we prove the curvature bound of $g_s$ in terms of that of $g_0$ and $g_1$
and the injectivity radius thereof.
\begin{theorem}\label{thm:Rs-bounds} Let $C_0,\, C_1$ and $C' = C'(C_0, C_1, r_0, r_1)$ be
as above. Then we have
\be\label{eq:Rs-bounds}
\|R_s\|_{C^0} \leq (1-s) C_0 + s C_1 + \frac{2 s(1-s)}{\min\{1, \lambda_P\} }(C')^2.
\ee
\end{theorem}
\begin{proof} The proof immediately follows from expressing the formula
\eqref{eq:Rs} after writing the formula in coordinates on normal neighborhoods of the uniform size 
$\epsilon = \epsilon(r_0,r_1) > 0$
and applying the estimates obtained from the above lemmata.
\end{proof}

By the same kind of reasoning after taking the derivatives of the identity
\eqref{eq:Rs}, we also derive the following uniform $C^1$ bound.

\begin{theorem}\label{thm:DRs-bounds}
Suppose $\max\{\|DR_{g_0}\|_{C^0}, \|DR_{g_1}\|_{C^0}\} \leq C_2$ for some 
$C_2 > 0$ in addition. Then there exists some $C'' = C''(C_0,C_1,r_0,r_1,C_2) > 0$ 
such that
$$
\max_{s \in [0,1]} \|DR_s\|_{C^0} \leq C''.
$$
\end{theorem}

\section{Lower bounds for the injectivity radii of the convex sum}

This is the central section of the present paper.
The goal of this section is to prove a uniform lower estimate of injectivity
radius for the convex sum.
For this purpose,  we will utilize some idea entering in the proof of a quantitative inverse function theorem.
(See the proofs of \cite[Corollary 2.5.6]{abraham-marsden-ratiu} or of \cite[Section 8]{christ:hilbert} which
are given \emph{under the full uniform $C^2$ bounds on the maps}.) 

For each given tame metric $g$, we isometrically identify
$(T_p M, g|_p)$ with $\R^n$ with the standard
inner product once and for all. We denote this identification map by
$$
I_p^{g}: \R^n \to T_p M
$$
by choosing an orthonormal frame $\mathscr{B}$
of $TM|_{B_{r_0}}^{g}(p)$ of $g$.

\begin{rem}\label{rem:cheeger-pinching}
Compare the map $I^g_p$ with the linear isometry, also denoted by $I: T_{\underline m}\underline M \to  T_mM$,
between two Riemannian manifolds $M, \, \underline M$ in \cite{cheeger:pinching}. The map is then composed with an exponential map to study some continuity property of the composition
$$
 I \circ \exp_{\underline m}^{-1}: \underline M \to T_{m}M.
$$
One might regard the above map $I_q^g$ the same kind of map when $\underline M$ is the \emph{fixed} universal space $\R^n$.
\end{rem}

This map restricts to a diffeomorphism from
\be\label{eq:Up}
U_p:=(I_p^{g})^{-1}\left((\exp_p^{g})^{-1}(B_{r_0}^{g}(p))\right) \subset \R^n
\ee
to 
\be\label{eq:Vp}
V_p: = (\exp_p^{g})^{-1}(B_{r_0}^{g}(p)) \subset T_pM
\ee
for each $p \in M$ provided $r_0 < \iota_g$. We consider the map 
\be\label{eq:psipg}
\psi_q^g = \exp_q^g \circ I_q^g
\ee 
at each point $q \in M$. The following is easy to check but is a key lemma that plays a fundamental 
role in our application of an argument of the quantitative inverse function theorem.

\begin{lem}\label{lem:psiqg} 
Suppose that the metric $g$ is complete. Then the maximal domain of the map $\psi_q^g$ is $\R^n$.
\end{lem}
\begin{proof} This is an immediate consequence of the definition of $\psi_q^g = \exp_q^g \circ I_q^g$ 
from the completeness hypothesis of the metric
so that the exponential map $\exp_q^g$ is globally defined at all $q \in M$.
\end{proof}

Let $g_{\text{\rm rf}}$ be any given reference \emph{tame} metric 
and consider a constant 
\be\label{eq:r0}
0 < r_0 < \iota_{g_{\text{\rm rf}}}.
\ee
We fix this $r_0$ and a coordinate atlas of $M$ given by 
\be\label{eq:geodesic-atlas}
\left\{\left( (\psi_q^{g_{\text{\rm rf}}})^{-1},B^{g_{\text{\rm rf}}}_{r_0}(q)\right)\right\}_{q \in M};  \quad \psi_q^{g_{\text{\rm rf}}}: =
\exp_{q}^{g_{\text{\rm rf}}} \circ I_q^{g_{\text{\rm rf}}}
\ee
which provides the Gaussian normal coordinates $(x^1,\ldots, x^n) := (\psi_q^{g_{\text{\rm rf}}})^{-1}$
on the geodesic normal balls $B^{g_{\text{\rm rf}}}_{r_0}(q)$ associated to the metric $g_{\text{\rm rf}}$.
Without loss of generality,  we may assume that \emph{the geodesic normal ball
$$
B^{g_{\text{\rm rf}}}_{r_0}(q) =  (\psi_q^{g_{\text{\rm rf}}})^{-1}(B^{n}(r_0))
$$
is strongly convex for all $q \in M$ by shrinking $r_0$ further if necessary.}
Such a shrinking of $r_0> 0$ can be verified by the tame property of
the reference metric $g_{\text{\rm rf}}$. 
(See \cite[Appendix p.103]{cheeger-ebin} for its proof.)

For two different points $q, \, p$, we have the following diagram:

\medskip

\be
\xymatrix{&U_q \ar[ld]_{I_q^{g_0}} \ar@{^{(}->}[r]\ar[d]_{\psi_q^{g_0}}  
& {\R^n} \ar[ld]^{\psi_q^{g_0}} \ar[rd]_{\psi_p^{g_1}} & U_p \ar@{_{(}->}[l]\ \ar[d]^{\psi_p^{g_1}} 
\ar[rd]^{I_p^{g_1}}\\
V_q \ar@{^{(}->}[r]_{\exp_q^{g_0}} & M & {} & M & V_p
\ar@{_{(}->}[l]^{\exp_p^{g_1}}.
}
\ee
\medskip

\noindent (We note that the map $\psi_q^g$ is globally defined on whole $\R^n$ as long as the metric $g$
is complete.  Therefore the map restricts to an injective map on any open subset of the type $U_q$
defined above at every point $q \in M$ by the definition of geodesic normal ball $B_r^g(q)$.)

Recalling the definition of the injectivity radius function $\inj_g: M \to \R$
and the injectivity radius $\iota_g: = \inf_{x \in M} \inj_g(x)$,
 the following uniform lower estimate of the injectivity radii is the key result towards
 the proof of our convexity result.

\begin{theorem}\label{thm:iota-bounds} Suppose $\|R_{g_k}\|_{C^1} \leq C$ 
and $\iota_{g_k} \geq \epsilon_k$ for $k = 0, \, 1$. 
Consider the convex sum $g_s = (1-s)g_0 + sg_1$. Then
there exists some $\epsilon' = \epsilon'(C, \epsilon_0, \epsilon_1) > 0$
such that
$$
\iota_{g_s} \geq \epsilon' > 0
$$
for all $s \in [0,1]$.
\end{theorem}
As mentioned before, the main difficulty of the proof lies in the facts:
\begin{enumerate}
\item The set of $C^r$ tame metrics may not be convex with low regularity $1 \leq r  <3$.
\item The function $s \mapsto \iota_{g_s}$ is not a continuous function in \emph{strong} $C^k$ topology.
\end{enumerate}

Our strategy of proving the theorem is as follows:
Let $B \subset [0,1]$ be the set of $s \in [0,1]$ such that $\iota_{g_s} > 0$. By the given hypothesis, $B$ is nonempty.
Then we will show that $B$ is open and closed which will then show that $B=[0,1]$ by the connectedness of 
the interval $[0,1]$. After that, we will find a lower bound $\epsilon'>0$ of injectivity radius and finish the proof.

\subsection{Implication of curvature estimates: openness}

Let $s \in [0,1]$ for which $\iota_{g_s} > 0$. 
We will find some $\delta > 0$ and $\epsilon > 0$ for which
$\iota_{g_{s'}} > \epsilon$ for all $s' \in (s - \delta, s + \delta) \cap [0,1]$.
We consider the  initial value problem for the geodesic equation of the metric $g_{s'}$
\be\label{eq:geodesi-i}
\nabla_{\dot \gamma}^{g_{s'}} \dot \gamma = 0, \quad \gamma(0) = p
\ee
in terms of the Gaussian normal coordinates $(x^1, \cdots, x^n)$ at each $p$ and $s \in [0,1]$.
With respect to the associated canonical coordinates
$$
(x^1, \cdots, x^n, v^1,\cdots, v^n)
$$
of $(x^1, \cdots, x^n)$,
the equation becomes the equation of geodesic flow  of $g_s$ 
given by
\be\label{eq:geodesic-flow}
\begin{cases}\dot x^j = v^j, \\
\dot v^j = - \Gamma^j_{k\ell}(x^1,\cdots, x^n) v^kv^\ell.
\end{cases}
\ee
We recall that the geodesic flow is the Hamiltonian flow of the kinetic energy of $g_s$ on $TM$
which is globally defined on $TM$, \emph{provided the metric $g$ is complete}.
The equation \eqref{eq:geodesic-flow}
is the coordinate expression thereof in the canonical coordinates.
(See \cite{klingenberg} for detailed explanations of this point of view.)

\begin{rem}  It follows from the expression \eqref{eq:Gammakij} of the Christoffel symbols
that the local existence, uniqueness and continuity of solutions of \eqref{eq:geodesic-flow}
hold as long as
the metric $g$ is in the class of $C^{1,1}$ and uniform if there is a uniform bound on
$C^{1,1}$-norm as in the case \emph{when there are uniform bounds on the curvature and
the quasi-isometric ratio $A(g,g_{\text{\rm rf}})$ with respect to a given back-ground metric 
$g_{\text{\rm rf}}$}.
\end{rem}
The following lemma is standard, which states that the geodesic flow on $TM$ is
nothing but the Hamiltonian flow of $K_g$ for any Riemannian metric $g$.

\begin{lem} \label{lem:geodesic-flow}
Let $\phi_{K_g}^t$ be the flow of the vector field $X_{K_g}$ and $\pi:TM \to M$ the canonical
projection. Then
$$
\exp_p^g(v) = \pi \circ \phi_{K_g}^1(p,v)
$$
and the metric $g$ is complete if and only if the vector field $X_{K_g}$ is
complete.
\end{lem}

We apply this lemma to the present one-parameter family $\{g_s\}$ of metrics $g_s$.
We denote by $K_s : TM \to \R$ the kinetic energy Hamiltonian of $g_s$ with $s \in [0,1]$,
$X_{K_s}$ the associated Hamiltonian vector field and $\phi_{K_s}^t$ the associated
Hamiltonian flow which is nothing but the geodesic flow of the metric $g_s$. 
Since $g_s$ is complete, the Hamiltonian flow thereof is  well-defined 
for all time $t \in \R$, 
in particular a complete trajectory of the geodesic flow of each $g_s$
with the initial condition $(x(0) ,v(0))= ( p, v_0)$ exists for every $v_0 \in T_pM$.

Furthermore the curvature bound implies the following bounds.

\begin{lem}\label{lem:Gamma-bound}  Under the same hypotheses as above in the standing hypotheses of Section \ref{sec:Rs-bound}, 
there is a uniform constant $C'''> 0$ for which on 
$B_{r_0}^{g_{\text{\rm rf}}}(q)$
\be\label{eq:Gamma-bound}
\|\Gamma^j_{k\ell} \|_{C^0}, \quad  \left \|\Gamma^j_{k\ell;m}\right\|_{C^0} \leq C'''
\ee
holds for all $j, \, k,\, \ell, \, m$.
\end{lem}
\begin{proof} Since the first bound is already proved in the course of the proof of
Proposition \ref{prop:|D|}, we have only to prove the second bound.
For this, we differentiate \eqref{eq:gij} twice and obtain
$$
g_{ij;km}(p) =   \frac13 R_{ikm j}(p)
$$
in any geodesic normal coordinates at each $p$.
This provides the bound 
$$
\|g_{ij;km}\|_{C^0} < C  \quad \text{\rm  on }\, B_{r_0}^{g_i}(p)
$$
for some constant $C> 0$ depending only on the curvature bound and 
the injectivity radius lower bound of $g_i$ for $i = 0, \, 1$.

Therefore since $\Gamma_{ij;m}$ is a sum of polynomials of $g_{ij;km}$, $g_{ij}$
and $g^{ij}$ of degree at most 3,  we have finished the proof again by applying Lemma \ref{lem:cheeger}.
\end{proof}

 We now consider the parameterized exponential map
\be\label{eq:psii}
(s,t,v) \mapsto \exp_p^{g_s}(t v) = \pi \circ \phi_{K_{s}}^t (p,v).
\ee
By definition, we have the (space) tangent map of $\exp_p^{g_s}: T_pM \to M$
$$
d \exp_p ^{g_{s}}= d \pi d\phi_{K_s}^t
$$
for each $s$.
We recall $d\exp_p^{g_s}(0): T_p M \to T_p M $ is the identity map
for all $s \in [0,1]$.

\begin{prop}\label{prop:one-one} Suppose $\iota_{g_s} > 0$ for an $s \in [0,1]$.
There exists a sufficiently small $\delta,  \, \epsilon > 0$ such that 
$$
\iota_{g_{s'}}> \epsilon
$$
for all $s' \in (s-\delta,s+\delta) \cap [0,1]$. 
\end{prop}
\begin{proof}
By the standing hypothesis, at each point $p \in M$, there exists an atlas of the form
$$
\left\{\left((\psi_p^{g_s})^{-1}, B_{\iota_{g_s}}^{g_s}(p)\right)\right\}_{p \in M}.
$$ 
We denote by 
$$
(x_1, \cdots, x_n) = (\psi_p^{g_s})^{-1}
$$  
the associated geodesic normal coordinates of $g_s$ on $B_{\iota_{g_s}}^{g_s}(p)$.
For any $s'\in [0,1]$, we may apply Lemma \ref{lem:cheeger} and get $\kappa_{s'}>0$ so that
$$
B_{\kappa_{s'}}^{g_{s'}}(p)\subset B_{\iota_{g_s}}^{g_s}(p).
$$
Now let $r_s:=\inf_{s'\in[0,1]}\kappa_{s'}$. Then $r_s>0$ and 
$\exp_{p}^{g_s}\circ I_p^{g_{s}}$ is a diffeomorphism of $B^n(\iota_{g_s})$ onto its image $B_{\iota_{g_s} }^{g_s}(p)$ in $M$.
Also, the inclusion
$$
 (\exp_p^{g_{s'}} \circ I_p^{g_{s'}} )(B^n(r_s)) \subset (\exp_p^{g_s} \circ I_p^{g_{s}})(B^n(\iota_{g_s}))$$ 
 implies that the composition
$$
 (\exp_p^{g_s} \circ I_p^{g_{s}})^{-1} \circ (\exp_p^{g_{s'}} \circ I_p^{g_{s'}} ) 
 : B^n(r_s)  \subset \R^n \to  B^n(\iota_{g_s}) \subset \R^n
$$
is defined for all $s' \in [0,1]$ \emph{by the completeness of $g_{s'}$}. The latter also implies
the exponential map $\exp_p^{g_{s'}}$ is defined on whole $T_pM$.

\begin{lem}\label{lem:<epsilons}Let $s$ be given and let $r_s> 0$ be as above.
There exists a sufficiently small $\epsilon_s> 0$ and in turn a $\delta > 0$ such that
for all $s' \in (s-\delta,s+ \delta)\cap [0,1]$
$$
\|(\psi_p^{g_s})^{-1} \circ \psi_p^{g_{s'}}  - \id\|_{C^1} < \epsilon_s
$$
on $B^n(r_{s})$ where $\epsilon_s$ can be made as small as we want by taking $\delta > 0$ small.
\end{lem}
\begin{proof} We mention that  $g_{s'} \to g_s$ in local $C^\infty$ topology as $s' \to s$.
In particular 
$$
(\psi_p^{g_{s}})^{-1} \circ \psi_p^{g_{s'}} \to \id
$$
on $\overline B^n(r_s) $ in $C^\infty$ topology and so in $C^1$-topology.
The lemma follows from this.
\end{proof}

In particular the map
$$
 (\psi_p^{g_s})^{-1} \circ \psi_p^{g_{s'}} = : \Phi_{ss'}
$$
is a homeomorphism (and so a diffeomorphism) from $B^n(r_s)$
onto its image $\Phi_{ss'}(B^n(r_s))$. By letting $\delta > 0$ smaller  if necessary, we may assume
that 
\be\label{eq:subset}
\overline{B^n(r_{s'})} \subset \Phi_{ss'}(B^n(r_s))
\ee
where we can choose $r_{s'} = \frac{r_s}{2}$ for every $s' \in (s-\delta, s+\delta)\cap [0,1]$ 
which does not depend on $s'$.

Then the map
\beastar
(\psi_p^{g_{s'}})^{-1} & = & (\exp_p^{g_{s'}} \circ I_p^{g_{s'}} )^{-1}
= ( I_p^{g_{s'}} )^{-1} \circ  (\exp_p^{g_{s'}} )^{-1} \\
& = &  ( \Phi_{ss'})^{-1}  \circ (\psi_p^{g_s} )^{-1}
\eeastar
is a well-defined diffeomorphism of $B^n(r_s/2)$ onto its image in $M$ for all $s' \in (s-\delta,s+\delta)\cap [0,1]$.

By \eqref{eq:subset}, we have shown that $\exp_p^{g_{s'}}: I_p^{g_{s'}}(B^n(r_s/2)) \to M$
is a diffeomorphism onto its image. In particular, we have proved that
$$
\text{\rm inj}_{g_{s'}}(p) \geq \frac12 r_s
$$
for all $p \in M$ provided $s' \in (s- \delta, s+\delta) \cap [0,1]$ for a sufficiently small $\delta > 0$.
This proves the proposition (and hence the openness of $B$) by setting $\epsilon_s = \frac{r_s}{2}$.
\end{proof}

\subsection{Uniform injectivity radii lower bound near a limit point: closedness}
\label{subsec:closedness}

We know that the map
\be\label{eq:psipg2}
\psi_p^{g}: = \exp_p^g \circ I_p^g: \R^n \to M
\ee
is well-defined by Lemma \ref{lem:psiqg}. 

\begin{rem}\label{rem:upshot-psipg}
The upshot of considering this map is to 
standardize the geodesic normal balls whose centers move around in $M$ to
the \emph{fixed} ball centered at the origin of $\R^n$. By doing so, we can consider
the family $\{\psi_p^g\}_{g, p}$ whose domain is a fixed ball in the Euclidean space $\R^n$.
These maps will be especially useful later when we try to control the injectivity radius 
\emph{as  the centers of the geodesic balls escape to infinity}. 
\end{rem}

Let $s_i \in B$ be
a sequence converging to $s_0 \in [0,1]$.  We need to show $\iota_{g_{s_0} }> 0$
which we will prove by contradiction.
Suppose to the contrary that 
\be\label{eq:standing-hypothesis}
\iota_{g_{s_0}} = 0.
\ee

We start with the following obvious lemma.

\begin{lem} There exists a constant $A_0\geq 1 $ such that
$$
\frac1{A_0} g_{s_0} \leq g_{s_i} \leq A_0 g_{s_0}
$$
and $\|R_{g_{s_i}} - R_{g_{s_0}}\|_{C^0} \to 0$ as $i \to \infty$.
\end{lem}

At this point, there are two cases to consider:
\begin{enumerate}
\item The case where $\iota_{g_{s_i}} > \epsilon_1 > 0$ for all $i$,
\item The case where there is a subsequence, still denoted by $s_i$, such that
$\iota_{g_{s_i}} \to 0$.
\end{enumerate}

\medskip

\noindent{\bf Case (1):} An examination of the proof of Proposition \ref{prop:one-one}
shows that the choice of $\delta> 0$ therein depends only on $\epsilon> 0$, the quasi-isometric
ratio $A(g_{s'},g_s)$. Therefore we can choose $\delta > 0$ independently of $i$'s 
by applying the same argument to all $g_{s_i}$ for the $\epsilon_1 > 0$ as in the proof of openness.
Therefore if $i$ is sufficiently large, $s_0 \in (s_i - \delta, s_i + \delta)$ 
 by the standing hypothesis \eqref{eq:standing-hypothesis}. This proves
$\iota_{g_{s_0}} > 0$. Therefore this case is ruled out.

\medskip

\noindent{\bf Case (2):} 
To proceed further, we will need a quantitative version of
inverse function theorem or rather some idea of its proof 
in some circumstance.
For readers' convenience, we recall the theorem stated in
 \cite[Corollary 2.5.6]{abraham-marsden-ratiu} below \emph{under the $C^2$ bounded assumption}
 of the relevant map $F$.

Let $g$ be any metric in the given sequence $\{g_{s_i}\}$ of  complete Riemannian metrics on $M$
which are  of bounded curvature, \emph{but a priori have no  uniform lower bound away from zero 
for  the  injectivity radii}.  

By the curvature bound, 
Case (2) means that there is a pair of geodesics 
$\gamma_{i,p_i}^\pm(t) = \exp_{p_i}(t v_i^\pm)$ such that 
\begin{enumerate}
\item $|v_i^\pm| = 1$ and $v_i^+ \neq v_i^-$ contained in $T_{p_i}M$, and
\item they satisfy
\be
\gamma_{i,p_i}^+(\ell_i^+) = \gamma_{i,p_i}^-(\ell_i^-),
\ee
with $\ell_i^\pm \to 0$ as  $i \to \infty $. 
\end{enumerate}
 (See \cite[Lemma 5.6]{cheeger-ebin} for example.)
Then we have $s_i \to s_0$ and $\text{\rm inj}^{g_{s_i}}(p_i) \to 0$ as $i \to\infty$.

We now quote the following quantitative inverse function theorem from 
\cite[Corollary 2.5.6]{abraham-marsden-ratiu} in which
the theorem is stated in the general setting of Banach manifolds.
(See also \cite[Section 8]{christ:hilbert} for an essentially same statement in the finite
dimensional case.)

\begin{theorem}[Inverse Function Theorem]\label{thm:IFT}
Let $F$ denote an arbitrary $C^2$ map from a convex open ball $U \subset \R^n$ to $\R^n$,
$x_0 \in U$ and $DF(x_0)$ is an isomorphism. We denote by $D^n$ and by $B^n$ a ball in the domain
and in the codomain of the map $F$ respectively. Write
$$
L = \|DF(x_0)\|, \, M = \|DF(x_0)^{-1}\|.
$$
Assume $\|D^2F(x)\| \leq K$ for $x \in D^n_R(x_0)$ and $\overline D^n_R(x_0) \subset U$.
Let
\beastar
R_1&  = & \min\left\{ \frac1{2KM},R\right\} \\
R_2 & = & \min\left\{\frac1R_1, \frac1{2M(L + KR_1)}\right\}, \quad R_3 = \frac{R_2}{2L}.
\eeastar
Then
\begin{enumerate}
\item $F$ maps the ball $D_{R_2}^n(x_0)$ diffeomorphically onto an open set containing the ball
$B_{R_3}^n(F(x_0))$. In particular, $F$ is one-to-one thereon.
\item For $y_1, \, y_2 \in  \overline B^n_{R_3}(F(x_0))$, we have
$$
\|F^{-1}(y_1) - F^{-1}(y_2)\| \leq 2L \|y_1 - y_2\|.
$$
\end{enumerate}
\end{theorem}

We would like to apply Theorem \ref{thm:IFT}
to the sequence of maps given by
\be\label{eq:Fi}
F_i(x): = \Phi_{0s_i}(x) = (\psi_{p_i}^{g_0})^{-1}( \psi_{p_i}^{g_{s_i}} (x))
\ee
as follows. We  recall from \eqref{eq:psipg} that the maps
$$
 \psi_{p_i}^{g_{s_i}} = \exp_{p_i}^{g_{s_i}} \circ I_{p_i}^{g_{s_i}}
$$ 
 have the full $\R^n$ as their domains and the full $M$ as its codomains, while
the map $(\psi_{p_i}^{g_0})^{-1}$ has  $B_{r_0}^{g_{s_i}}(p_i)$ as its domain and 
 $$
 U_i: =(\psi_{p_i}^{g_0})^{-1}(B_{r_0}^{g_{s_i}}(p_i)) \subset \R^n
 $$
 as its codomain. This $U_i$  is an open neighborhood of $0 \in \R^n$. Combining these,
 we conclude that $F_i$ is defined on 
 $$
(I_{p_i}^{g_{s_i}})^{-1}(B_{r_0}^{g_{s_i}}(p_i)) \subset \R^n
$$
whose image is contained in $U_i \subset  \R^n$. Here both domains and codomains are 
neighborhoods of $0 \in \R^n$ respectively.

Since we have $g_{s_i} \to g_{s_0} = g$  in $C^\infty$ topology 
and $d((\psi_{p_i}^{g_{s_i}})^{-1}\circ  \psi_{p_i}^{g_{s_i}})(0) = \id$ by Lemma \ref{lem:<epsilons},
we can find a common domain, say, $B^n(r_0/2)$ such that
$$
(I_{p_i}^{g_{s_i}})^{-1}(B_{r_0}^{g_{s_i}}(p_i)) \supset B^n(r_0/2)
$$
and that we have
$$
U_i  \subset B^n(r_0)
$$
for all sufficiently large $i$'s: Here the constant $r_0 > 0$ is the one chosen in 
\eqref{eq:r0} depending only on the reference metric $g_{\text{\rm rf}}$ after
suitably shrinking it. In summary, we have achieved the following for the maps $F_i$:
\begin{itemize}
\item We have well-defined maps $F_i: B^n(r_0/2) \to B^n(r_0) $ for all
suffciently large $i$'s,
\item We have a uniform bound $\|D^2 F_i\| < K$ for some $K > 0$ independent of $i$.
\end{itemize}
This enables us to apply Theorem \ref{thm:IFT} to the maps $F_i$.

\begin{lem}\label{lem:IFT-condition}
All $F_i$ satisfy the hypotheses of
Theorem \ref{thm:IFT} with uniform constants $L, \ M$ and $R$.
\end{lem}
\begin{proof}
The proof is an adaptation of Cheeger's proof of \cite[Lemma 4.3]{cheeger:finiteness}.
We postpone the detail of the proof till Appendix \ref{sec:IFT-condition}. This is another place
where the bound for the derivative of the curvature and so the $C^3$ bound for the metrics 
explicitly enter. (See Lemma \ref{lem:2nd-derivative}.)
\end{proof}

This implies that $F_i$ is injective on $B^n(C' r_0) $
for some sufficiently small $C' > 0$ independent of $i$'s. This in turn implies that the maps
$\psi_{p_i}^{g_{s_i}}$ are injective and so are $\exp_{p_i}^{g_{s_i}}$ on $B_{p_i}^{g_{s_i}}(C' r_0)$.

On the other hand by the standing hypothesis \eqref{eq:standing-hypothesis},
the equality
$\gamma_{i,p_i}^+(\ell_i^+) = \gamma_{i,p_i}^-(\ell_i^-)$ is equivalent to
$$
\exp_{p_i}(\ell_i^+v_i^+) = \exp_{p_i}(\ell_i^-v_i^-)
$$
with $\max\{\ell_i^+, \ell_i^-\} < C' r_0$ eventually as $i \to \infty$ which contradicts to
the injectivity of $\exp_{p_i}$ on $B_{p_i}(C' r_0)$.
Therefore Case (2) cannot occur either if we choose $r_0$ in $0 < r_0 < \frac14$ sufficiently small.

Therefore we have proved $\iota_{g_{s_0}} > 0$ by contradiction.
In conclusion, we have shown $B = [0,1]$.

\subsection{Uniform injectivity radii lower bound over $[0,1]$}
\label{subsec:uniform_lower_bound}
By Proposition \ref{prop:one-one}, for every $s\in B=[0,1]$, there exist $\delta(s),\,\epsilon(s)>0$ such that 
$$
\iota_{g_{s'}}> \epsilon(s)
$$
for all $s' \in (s-\delta(s),s+\delta(s)) \cap [0,1]$. 
Note that $\{(s-\delta(s),s+\delta(s)) \cap [0,1]\}_{s\in[0,1]}$ is an open cover of $[0,1]$.
Since  $[0,1]$ is compact, it has a finite open subcover, say,
$$
\{(s_1-\delta(s_1),s_1+\delta(s_1)) \cap [0,1],\dots,(s_m-\delta(s_m),s_m+\delta(s_m)) \cap [0,1]\}. 
$$
This implies
$$
\iota_{g_{s}}> \min\{\epsilon(s_1),\dots,\epsilon(s_m)\}>0
$$
for all $s\in[0,1]$. Then $\min\{\epsilon(s_1),\dots,\epsilon(s_m)\}$ is the desired lower bound $\epsilon'$ and we finished the proof of
Theorem \ref{thm:iota-bounds}.

\begin{rem}
 It is an interesting question whether one can explicitly estimate the injectivity radius
of convex sum $g_s$ in terms of the curvature bounds and the injectivity radii of $g_0, \, g_1$.
 There have been studies of injectivity radii
under the curvature estimates and the estimate of volume growth of geodesic balls in the literature.
(See \cite{cheng-li-yau}, \cite{CGT} to name a couple.)
These articles provide the estimates of the decay rate of injectivity radius of a
point as the point diverges to infinity but no uniform lower bound away from zero.
This is the reason why we have used the argument by contradiction as above
since there is no mention of volume growth of $g_0$ or $g_1$ in our circumstance.
In addition, there does not seem to be an easy way of formulating such a volume growth
 in our circumstance either, if possible at all.
\end{rem}

\section{Contractibility of the set $\Riem_{\mathfrak T}(M)$ of $C^r$-tame metrics for $r \geq 3$}
\label{sec:contractibility-g}

In the present section, we will prove Theorem \ref{thm:contractibility-intro}.

We first recall the definition of strong $C^k$ topology of $\Riem(M)$.
Let $g_{\text{\rm rf}} \in {\mathfrak T}$ be a fixed reference metric which will be
used for the study of $C^r$ topology on $\Riem_{\mathfrak T}(M)$. We also fix the associated
Levi-Civita connection of $g_{\text{\rm rf}} \in {\mathfrak T}$ denoted by $\nabla$. 
We will denote by $D$ the associated
covariant derivative applied to general tensor fields on $M$. We use 
$g_{\text{\rm rf}} \in {\mathfrak T}$ and $D$ to define the $C^k$-norms of general
tensor fields on $M$. In particular, we will focus on the metrics in the given quasi-isometry class $\mathfrak T$
\be\label{eq:ginRiemTM}
g \in \Riem_{\mathfrak T}(M).
\ee
A neighborhood basis element of $C^k$ (resp. $C^\infty$) topology
of $\Riem_{\mathfrak T}(M)$  (with respect to $g_{\text{\rm rf}}$)
at a given metric $g_0$ is given by the set of metrics 
$$
\CB_{\{\epsilon_i\}}(g_{0}) = \{g \in \Riem_{\mathfrak T}(M) \mid \|g-g_0\| < \epsilon_0, \ldots, 
\|D g\| \leq \epsilon_1, \ldots \|D^k g\| < \epsilon_r\}
$$
for a sequence $\epsilon_0, \ldots, \epsilon_k$ (resp. $\epsilon_0, \ldots, \epsilon_i, \ldots $) with $\epsilon_k > 0$.
Here $\|\cdot\|$ denotes the strong $C^0$-norm  with respect to the reference metric $g_{\text{\rm rf}}$
 of the space of sections
$$
\Gamma(\text{\rm Sym}^2(TM)) \supset \Riem_{\mathfrak T}(M)
$$
of symmetric 2-covariant tensor fields and is given by
$$
\|S\|: = \sup_{x \in M} |S(x)|
$$
for $S \in \Gamma(\text{\rm Sym}^2(TM))$.

Now we are ready to give the proof of the following 
\begin{theorem}[Contractibility in strong $C^k$ topology]\label{thm:Cr-contractibility} Let $r \geq 3$.
Let $\mathfrak T$ be any quasi-isometry class 
of $C^r$-tame Riemannian metrics on $M$ and $g_{\text{\rm rf}} \in \Riem_{\mathfrak T}(M)$. 
Then the function
\be\label{eq:L}
 L: [0,1] \times \Riem(M) \to \Riem(M)
\ee
defined by $L(s,g) = (1-s) g + s g_{\text{\rm rf}}$ restricts to
$$
 L: [0,1] \times \Riem_{\mathfrak T}(M) \to \Riem_{\mathfrak T}(M)
$$
as a continuous map, and so $\Riem_{\mathfrak{T}}(M)$ is contractible in strong $C^k$ topology
for any finite $0 \leq k < \infty$.
\end{theorem}
\begin{proof} We have already shown strict convexity of $\Riem_{\mathfrak T}(M)$ 
which implies the first statement. 
Therefore it is enough to prove the continuity of the map $L = L(s,g)$ on
$$
[0,1] \times \Riem(M).
$$
\emph{Here we highlight  the factor $\Riem(M)$}:
the aforementioned convexity enables us to work with  the factor $\Riem(M)$
is used for the study of continuity, not $\Riem_{\mathfrak T}(M)$, in the rest of the proof.

Let $(s_0, g_0) \in [0,1] \times \Riem(M)$, and denote by $g_{s_0}$ the metric given by
$$
g_{s_0}: = L(s_0,g_0) = (1-s_0) g_0 + s g_{\text{\rm rf}} 
$$
for the simplicity of notation. Recall that any neighborhood basis element 
of the strong $C^r$ topology of $\Riem(M)$ at $g_{s_0}$ is given by the collection
$$
\CB = \CB_{\{\epsilon_i\}}(g_{s_0}) 
$$
associated to the set $\{ \epsilon_i\} = \{\epsilon_i\}_{0 \leq i \leq r}$ of constants $\epsilon_i > 0$.
Therefore we need to prove that for each given such $\CB$ the preimage
$$
L^{-1}(\CB)
$$
is open in $[0,1] \times \Riem(M)$. 

By the definition of product topology of $[0,1] \times \Riem(M)$,
 we need to show that there exists $\delta > 0$ 
and a neighborhood basis element $\CB_{\{\delta_i\}}(g_0)$ of $g_0$ with $\delta_i$
 for $i = 0, \cdots, k$ such that 
$$
L((-\delta + s_0, \delta + s_0) \times \CB_{\{\delta_i\}}(g_0)) \subset \CB.
$$
In other words, we need to find $\delta$ and $\delta_i$ for $i = 0, \cdots, k$ such that
$$
\|D^k (L(s,g) - L(s_0,g_0))\|_{B_\alpha(r)} \leq \epsilon_k, \quad g_{s_0} = L(s_0,g_0)
$$
for all $s \in (-\delta + s_0, \delta + s_0)$ and $g \in \CB_{\{\delta_i\}}(g_0)$.
We have
\beastar
D^k (L(s,g) - L(s_0,g_0)) & = & D^k (L(s,g) - L(s,g_0)) + D^k (L(s,g_0) - L(s_0,g_0))\\
& = & (1-s)D^k(g - g_0) + (s-s_0)D^k(g_0).
\eeastar
By the assumption from \eqref{eq:ginRiemTM}
that $g_{s_0} \in {\mathfrak T}$ and $C^3$-tame to $g_{\text{\rm rf}}$, we can 
apply Theorem \ref{thm:iota-bounds} and cover  $M$ by an atlas 
$\{B_\alpha(r)\}$ of balls of uniform size of radii $r = r(B_\alpha) \geq \varepsilon > 0$
for some $\varepsilon > 0$. We may choose them so that $\{B_\alpha(r/2)\}_\alpha$ still covers $M$.

First we choose $\delta_i, \, i = 0, \cdots, k$ so that 
$$
\max \big\{ \|g - g_0\|_{B_\alpha(r)},  \{\|D^i(g - g_0)\|_{B_\alpha(r)}\}_{i=1}^k \big \}
\leq \frac12 \varepsilon 
$$
for all $\alpha$ and $g \in \CB_{\{\delta_i\}}(g_0)$.
Then we choose $\delta > 0$ so that
$$
\delta \max_{i =0, \cdots, k} \epsilon_i \leq \frac12 \varepsilon.
$$
Then for all $\alpha$, we obtain
\beastar
&{}& \|D^i (L(s,g) - L(s_0,g_0))\|_{B_\alpha(r)} \\
& \leq  &  (1-s)\| D^i(g - g_0)\|_{B_\alpha(r)}  + (s-s_0)\|D^i(g_0)\|_{B_\alpha(r)} \\
 & \leq &  \frac12 \varepsilon + \frac12 \varepsilon = \varepsilon, 
\eeastar
i.e., $L(s,g) \in \CB$  for all $g \in \CB_{\{\delta_i\}}(g_0)$ and for all $s$ with $|s-s_0| < \delta$. This 
finishes the proof of the continuity of $L$.
\end{proof}

\begin{rem}\label{rem:direct-limit-topology}
This proof clearly shows why it cannot be applied to the usual $C^\infty$ case of infinite regularity,
but the proof can be adapted to the \emph{direct limit strong $C^\infty$ topology} introduced in
Appendix \ref{sec:direct-limit}.
\end{rem}

\section{No cusp-developing under the $C^3$-continuation of tame metrics}

In this section, we prove Theorem \ref{thm:noncollapsing-intro} which we restate here.

\begin{theorem}\label{thm:no-cusp} Consider any continuous family $g_s$
in strong $C^3$ topology of complete Riemannian metrics with
$$
\|R_{g_0}\|_{C^0}, \, \|DR_{g_0}\|_{C^0}  < C, \quad \iota_{g_0} > \epsilon.
$$
Then there is a constant $C' = C'(C, \{g_s\}), \, \epsilon' = \epsilon'(C, \epsilon, \{g_s\}) > 0$
such that
\be\label{eq:boundsforg_s}
\inf_{s \in [0,1]} \iota_{g_s} > \epsilon'
\ee
\end{theorem}
\begin{proof} We first note that since $g_s$ is $C^3$-continuous and $[0,1]$ is compact, the uniform curvature
bound immediately follows by continuity and compactness. The quasi-isometric equivalence
 between $g_0$ and $g_1$ is also
an immediate consequence thereof. It remains to show the
injectivity radii lower bound. 

Let $B \subset [0,1]$ be the subset consisting of $s$'s for which $\iota_{g_s} > 0$. 
The openness of $B$ immediately follows from the curvature derivative bound and
the $C^3$-continuity of the map $s \mapsto g_s$. 
The proof of the lower bound goes exactly the same as the proof of closedness
in the previous section, which relies on the completeness of 
the metrics and the quantitative inverse function theorem, Theorem \ref{thm:IFT},
especially the injectivity statement \emph{under the uniform $C^3$ bound}.

Again we would like to apply Theorem \ref{thm:IFT}
to the sequence of metrics $g_{s_i} \to g_{s_0} = g$ by considering the associated maps
$$
F_i: B^n(r_0/2) \to B^n(r_0)
$$
with fixed domain and codomain given by
\be\label{eq:Fi2}
F_i(x): = (\psi_p^{g_0})^{-1}( \psi_p^{g_{s_i}} (x))
\ee
on $B^n(r_0/2)$ as before. By the given hypothesis of the derivative curvature bound and 
the $C^3$-continuity of the map $s \mapsto g_s$, we can again apply 
Lemma \ref{lem:IFT-condition}.
This implies that $F_i$ is injective on $B^n(C' r_0) $
for some $C' > 0$ independent of $i$'s which again  implies that 
$\exp_p^{g_{s_i}}$ is injective on $I_p^{g_{s_i}}(B^n(C' r_0)) \subset T_pM$.

Once we have achieved this far, exactly the same proof by contradiction as 
in the proof of Theorem \ref{thm:iota-bounds}, especially the closedness part
thereof given in the previous Subsection \ref{subsec:closedness} applies to 
derive a contradiction which in turn implies that $B = [0,1]$.
This finishes the proof.
\end{proof}

Our proof of the existence of a uniform injectivity lower bound is not direct in that
it does not provide an explicit bound in terms of the given geometric bound.
In our proof, it is essential to consider a one-parameter family
$g_s$ for $s \in [0,1]$ of complete metrics of uniformly bounded curvature.
In this sense our result is a deformation result.
The hypothesis that \emph{one of the initial metrics $g_0$ and $g_1$ has
bounded curvature and injectivity lower bound away from zero}, which
provides a covering of the manifold by a uniform size geodesic normal
coordinate charts. See \cite{cheeger:finiteness} for the importance of covering the manifold.

\begin{rem} According to the injectivity lower bound formula 
from \cite[Equation (4.23)]{CGT}
it will be enough to have a \emph{lower bound} for the volume growth
$V_r^{g_s}(p): = \vol_{g_s}(B_{r}^{g_s}(p))$ over $s \in [0,1]$. (See \cite[Section 4]{CGT}
for some detailed discussion on the relationship between the volume lower bound
and the injectivity lower bound on complete open Riemannian manifold.)
It would be interesting to examine whether the uniform lower bound of the
volume $V_r^{g_s}(p)$ exists for some choice of $p \in M$ and $r > 0$ depending on the
given one-parameter family $\{g_s\}$ starting from $g_0$ in our situation.
\end{rem}

\part{Application to large-scale symplectic topology}

\section{{$\CJ_\omega$} is not connected in strong $C^\infty$ topology}
\label{sec:not-connected}

We first recall the standard definition of almost complex structures tame to
a symplectic form in general as formulated by Sikorav \cite{sikorav:holo}.

\begin{defn}\label{defn:tame-J}
Let $(N,J,g)$ be an almost complex manifold equipped with a Riemannian metric
$g$. We say the triple is \emph{$C^r$-tame} if the following hold:
\begin{enumerate}
\item $g$ is complete.
\item $g$ has bounded curvature and injectivity radius bounded away from zero.
More precisely, there exist constants $C_0 > 0$ and $\varepsilon_0 > 0$ such that
its curvature $R_g$ satisfies $\|R_g\|_{C^{k-2}} \leq C_0 < \infty$ and
and its injectivity radius $\iota_g$ satisfies $\iota_g \geq \varepsilon_0 > 0$.
\item $J$ is uniformly continuous with respect to $g$.
\end{enumerate}
\end{defn}

Condition (3) is usually rephrased into the statement that there is a tame metric $g$ and constant $C> 0$
such that \eqref{eq:A-quasiisometry} holds for the pair of metrics with $(g,g_J)$ and
with $A = C$.

\begin{rem}
\begin{enumerate}
\item 
It appears that Sikorav's formulation \cite{sikorav:holo} starting from a
Riemannian metric should be the way how one should introduce
the concept of \emph{almost complex structures tame to a symplectic form}
in strong $C^\infty$ topology \emph{by remembering its quasi-isometry class 
of the associated metric until the end, not forgetting away along the way}.
\item As shown in Part 1 of the present paper,  we need at least the $C^3$-tameness
to be able show that the set of tame almost complex structures is contractible.
See Subsection \ref{subsec:gJ-tame} for the reason why.
\end{enumerate}
\end{rem}
\emph{Since the regularity is not the main issue of the present part, we will
always assume that the triple $(N,J,g)$ is $C^\infty$-tame in the rest of the paper
without further mentioning.}

Recall that the notion of $J$-holomorphic curves $u$ is nothing but the same as
the almost complex curves $u: (\Sigma,j) \to (N, J)$ which can be defined
 for any almost complex manifold.

\begin{defn} Let $(M,\omega)$ be a symplectic manifold. An almost complex structure $J$ on
$M$ is \emph{compatible to $\omega$} if the following hold:
\begin{enumerate}
\item The bilinear form $g_J: = \omega(\cdot, J \cdot)$ is  symmetric positive definite.
\item The triple $(M,J,g_J)$ is tame in the sense of Definition of \ref{defn:tame-J}.
\end{enumerate}
We denote by $\CJ_\omega$
the set of $\omega$-compatible almost complex structures and call the metric $g_J$
an $\omega$-tame metric.
\end{defn}

More generally, we say $J$ is \emph{tame to $\omega$} if Condition (1) is weakened by
omitting the symmetry of the bilinear form $\omega(\cdot, J \cdot)$: To any tame
$\omega$, we have a canonically defined Riemannian metric given by symmetrizing
the bilinear form, again denoted by $g_J = g_{(\omega,J)}$,
$$
g_J(v_1,v_2): = \frac{\omega(v_1,Jv_2) + \omega(v_2,Jv_1)}{2}
$$
Recall that $\CJ_\omega$ carries the structure of an infinite dimensional
smooth Fr\'{e}chet manifold the tangent space of which can be written as
\bea\label{eq:TJCJ}
T_J \CJ_\omega & = & \{ B \in \Gamma(\text{\rm End}(E)) \mid
 BJ + JB = 0,  \, \omega(B(\cdot), J (\cdot))
 + \omega(J(\cdot), B(\cdot)) =0 \}
 \nonumber\\
\eea
where the second equation means nothing but that $B$ is
a symmetric endomorphism of the metric $g_J=\omega(\cdot, J \cdot)$
on $\xi$. (See \cite{floer:unregularized}.)

With this preparation, Gromov's proof of connectedness of $\CJ_\omega$
for the compact case goes as follows.
By definition, $\CJ_\omega$ can be expressed as the space of smooth sections of the fiber bundle
\be\label{eq:Slambda}
S_\omega \to M
\ee
whose fiber is given by $S_{\omega,x}$ which is isomorphic to
$$
 S(\R^{2n}): = \{J_0 \in \text{\rm Aut}(\R^{2n}) \mid
J_0^2 = -\text{\rm Id},  \, - J_0 \Omega_0 J_0 \, \text{\rm is positive definite}  \}
 $$
where $\Omega_0$ is the matrix associated to the standard symplectic
bilinear form on $\R^{2n}$. When identified with $\R^{2n} \cong \C^n$,
$\Omega_0$ corresponds to the complex multiplication by $\sqrt{-1}$ with
the identification $\R^{2n} \cong \C^n$. Then this set is contractible.
Gromov then concludes in \cite[Corollary 2.3]{gromov:invent} that
\emph{$\CJ_\omega$ is contractible}.

For the noncompact case, the same proof still applies in the
strong $C^3$ topology by the result from Part 1
as long as the metric $g_J: = \omega(\cdot, J\cdot)$ is
contained in the given quasi-isometric class of Riemannian metrics.
Therefore for the Fukaya category recently constructed in
\cite{choi-oh:construction} to be regarded as a symplectic invariant,
one must restrict to the deformations of $J$'s that are
continuous  in strong $C^r$ topology  with $3 \leq r < \infty$.

Now comes one natural question arises: \emph{Are two metrics tame to the given
symplectic form $\omega$ quasi-isometric?}
The answer to this question is simply \emph{no} as the following 2 dimensional example shows.

\begin{exam}\label{exam:2-dim}
Consider the plane $(\R^2, \omega_0)$ with the standard symplectic form
$\omega_0$. Consider two Riemannian metrics, $g_0$ the standard flat metric
on $\R^2$ and the other the metric $g_1$ of the \emph{cigar}
$$
\{(x,y,z) \mid x^2 + y^2 = 1, \, z \geq 0\} \cup \{(x,y,z) \mid x^2 + y^2 + z^2 =1,
\, z \leq 0\}  \subset \R^3
$$
with suitable smoothing thereof along the seam $z = 0$.

Obviously they are not quasi-isometric. Both have bounded curvatures and 
injectivity radii bounded away from zero. Denote by
$\omega_0 = \omega_{\text{\rm flat}}$ and $\omega_1 =\omega_{\text{\rm cigar}}$
the associated Riemannian area forms.
Since both have infinite volume, there is a diffeomorphism
$$
\psi: (\R^2,\omega_0) \to (\R^2,\omega_1)
$$
such that $\psi^*\omega_1 = \omega_0$ by \cite{Greene-Shiohama}.

Then the associated Riemannian metrics become
$$
\omega_0(\cdot, J_0 \cdot) = g_0, \quad \omega_0(\cdot, J_1 \cdot) = g_1,
$$
Therefore both $J_0$ and $J_1$ are tame to $\omega_0$ but their associated metrics
are not quasi-isometric.
\end{exam}
The main purpose of the present part is to identify a subset of the set $\CJ_\omega$
such that for any two $J_1, \, J_2$ in the subset the associated metrics $g_{J_1}$ and $g_{J_2}$ are
quasi-isometric.

We first recall the definition of strong $C^r$ topology of $\CJ_\omega$.
(See Appendix for the definition of strong $C^r$ topology (resp. $C^\infty$ topology) on the
space of sections of general fiber bundle.)
Once this class of subsets is identified, the proof of disconnectedness will
follow from the fact that the strong $C^\infty$ topology of
$C^\infty(M,S(\R^{2n}))$ is not path-connected. (See \cite{hirsch}.)

We will use the following standard notion in the comparison geometry for the study of this 
path-connectedness.
\begin{defn} \label{defn:quasiisometric-ratio}
For a given pair of Riemannian metrics $g_1, \, g_2$, we define the function 
$M_{g_1,g_2}: M \to \R_+$ given by
$$
M_{g_1,g_2}(x) = \sup_{0 \neq v \in T_x M} \left|\log\left(\frac{|v|_{g_2}}{|v|_{g_1}}\right)\right|
$$
and call the \emph{quasi-isometric ratio function of $g_1$ and $g_2$.}
\end{defn}
It follows that $M_{g_1,g_2}$ is a continuous function on $M$, and its definition
depends only on the $C^0$-values of $g_1, \, g_2$ on the tangent bundle $TM$.

We now prove the following.

\begin{prop} \label{prop:not-connected}
Let $(M,\omega)$ be a smooth noncompact surface
equipped with an area form of infinite area, and denote by
$\CJ_\omega$ the set of almost complex structures tame to $\omega$
equipped with strong $C^\infty$ topology. 
Then $\CJ_\omega$ is not path-connected.
\end{prop}
\begin{proof}
Denote by $\text{\rm Riem}_\omega (M)$ the set of $\omega$-tame Riemannian metrics 
$$
\text{\rm Riem}_\omega (M) = \{g_J \mid J \in \CJ_{\omega} \}
$$
similarly as defined in \eqref{eq:gJ}.
If $\CJ_\omega$ is path-connected, then $\text{\rm Riem}_\omega (M)$ also should be path-connected.
We will find a pair of disjoint connected components of $\text{\rm Riem}_\omega (M)$
which will then show that it is not connected.

Let $g_1,\,g_2$ be tame metrics. Consider the function $M_{g_1,g_2}$. 
Now suppose that $g_1,\,g_2$ are tame metrics which are not quasi-isometric.
Then we can find a sequence of points $x_1,x_2,...$ of $M$ so that 
$$
\lim_{n\to\infty}M_{g_1,g_2}(x_n)=\infty.
$$
Since two metrics on a compact set are always quasi-isometric, we may also assume that
there exists a compact exhaustion $C_1\subset C_2\subset \cdots$ of $M$ so that
\begin{itemize}
\item $C_m\subset \text{\rm Int}(C_{m+1})$ for all $m\in \N$ and
\item $\{x_n\}_{n\in\N} \cap C_m$ is finite for all $m\in \N$.
\end{itemize}

We now define an  equivalence relation  $\sim$ on $\text{\rm Riem}_\omega (M)$
by setting $g \sim h$ if and only if
\be\label{eq:n-to-infty}
\lim_{n\to\infty} \frac{M_{g,h}(x_n)}{M_{g_1,g_2}(x_n)}=0.
\ee
We can easily check that $\sim$ is an equivalence relation $\sim$ and so its equivalence classes
$\{E_\alpha\}_{\alpha \in I}$ for some indexing set $I$ partitions $\text{\rm Riem}_\omega (M)$ into a disjoint union
$$
\text{\rm Riem}_\omega (M) = \bigsqcup_{\alpha \in I} E_\alpha.
$$
We will show that each equivalence of  $\sim$ is open and
there are at least two nonempty equivalence classes, which will finish the proof.

\begin{lem} Each equivalence class is an open subset of $\text{\rm Riem}_\omega (M)$.
\end{lem}
\begin{proof} Pick $E_\alpha$ and  $g \in E_\alpha$. We will find   
 a strong basic open neighborhood $\CU$ of $g$ such that $h \sim g$
 i.e., \eqref{eq:n-to-infty} holds for all $h \in \CU$.  (see \cite{hirsch} for the definition
 of strong basic open neighborhoods).
 
For this purpose, let $\{(\varphi_i,U_i)\}_{i\in\Lambda}$ be a locally finite atlas of $M$ for some
countable indexing set $\Lambda$.
Note that
$$
\{\text{\rm Int}(C_2),\text{\rm Int}(C_3)\cap {C_1}^c,
\text{\rm Int}(C_4)\cap {C_2}^c,\dots\}
$$
is a  locally finite open cover of $M$ and taking intersection with any locally finite atlas
gives us a locally finite atlas of $M$. Moreover, the resulting atlas has
additional property that the intersection between $\{x_n\}_{n\in\N}$ and any chart is finite.
Therefore, we may assume that $K_\ell :=\{x_n\}_{n\in\N}\cap U_\ell$ is compact for all $\ell \in \Lambda$, 
and hence satisfies \eqref{eq:n-to-infty} for all $h \in \CU$.

Furthermore by the $C^3$-tameness, especially 
by the curvature bound and the injectivity lower bound established in Part 1,
we can choose the coordinate chart 
$$
\{(U_\ell,\varphi_\ell)\}_{\ell \in \Lambda}
$$
for some indexing countable set $I$
has a uniform size with uniform bound of the second derivatives of $\varphi_\ell$ and its inverse.

Then a strong basic neighborhood of $g$ in $\text{\rm Riem}_\omega (M)$
has the form
$$
\{h \in \text{\rm Riem}_\omega (M) \mid
\|\nabla^k h - \nabla^k g\| < \delta_k, \, k = 0, 1, \cdots \}
$$
for any sequence $\delta_k > 0$ in terms of the covariant derivatives $\nabla^k$ of 
the Levi-Civita connection of $g$, or equivalently the set with the defining inequality 
replaced by the norm
$$
 \sup_{\ell \in \Lambda} \sup_{ i \in \Z_{\geq 0}} \|D^k (h \circ \varphi_\ell^{-1})\|_{C^0(U_i)} < \delta_k
 $$
for all $k$. In particular, recalling that the definition of 
$M_{g,h}: M \to \R$ depends only on the $C^0$-values of $g, \, h$ 
on the tangent bundle $TM$, there exists $C = C(g; \{\delta_k\}_{0 \leq k\leq 2})> 0$ 
depending only on $g$ and $\{\delta_k\}_{0 \leq k\leq 2}$ such that
$$
\frac{1}{C}\leq M_{g,h}(x_n) \leq C
$$
for all $n$. This proves the pair $(g,h)$ satisfies
\eqref{eq:n-to-infty} and hence $h \in E_\alpha$. This proves that $E_\alpha$ is open.
\end{proof}

We obviously have
$$
\lim_{n\to\infty} \frac{M_{g_1,g_2}(x_n)}{M_{g_1,g_2}(x_n)}= 1 \neq 0.
$$
which implies $g_1\not \sim g_2$. Therefore there are at least two distinct equivalence
classes of $g_1$ and $g_2$. Therefore $\text{\rm Riem}_\omega (M)$ is not connected
which finishes the proof of the proposition.
\end{proof}
Note  it is well-known  that the set $\CJ_\omega$ is path connected in the weak $C^\infty$ topology.

\begin{rem}\label{rem:symplectization} There is a natural family of noncompact symplectic manifold,
the symplectization of a contact manifold, more specifically the product
$$
(Q \times \R, d(e^s \lambda)).
$$
  In this case, the way how the relevant
almost complex structures considered in the analysis of pseudo-holomorphic curves
on symplectization is to start from CR almost complex structures $J$ on
a contact manifold and consider the form of almost complex structures
$\widetilde J = J \oplus J_0$ in terms of the decomposition
$$
T(Q \times \R) = \xi \oplus \CV
$$
where $\xi$ is the contact distribution of $Q$ and
$$
\CV := \R \left\{R_\lambda, \frac{\del}{\del s}\right\}
$$
and $J_0$ is the unique almost complex structure on $\CV \cong \R^2$ satisfying
$J_0(\frac{\del}{\del s}) = R_\lambda$. Therefore in this case there is a choice of
almost complex structures that is essentially determined by the CR-almost complex structures $J$ 
on the given contact manifold $Q$.
If $Q$ is compact, the relevant topology is the $C^\infty$ topology of the set of $J$'s,
which also uniquely determines the relevant quasi-isometry class of the symplectization.

If $Q$ is noncompact as in the one-jet bundle one should tackle the similar
topological issue for the contact manifold $Q$ in the choice of
CR almost complex structures as done in \cite{oh:entanglement1} where
the notion of tame contact manifolds is introduced.
\end{rem}

\section{Quasi-isometry class $\mathfrak T$ and tame almost complex structures}
\label{sec:bilipschitz}

In this section, we introduce the set of \emph{quasi-isometrically tame} almost complex structures
denoted by $\CJ_{\omega;\mathfrak{T}}$ and study its (local) contractibility.

We start with the following definition.

\begin{defn}[$(\omega,\mathfrak{T})$-tame almost complex structures]
\label{defn:omegaT-tame}
 We call $J$ an $(\omega,\mathfrak{T})$-tame almost complex structure
 if the following hold:
 \begin{enumerate}
 \item It is $\omega$-tame.
 \item The Riemannian metric $g_J = \omega(\cdot, J \cdot)$ is in 
 the quasi-isometry class $\mathfrak{T}$.
 \end{enumerate}
 \end{defn}

A priori  the subset $\CJ_{\omega;\mathfrak{T}} \subset \CJ_{\omega}$
is not known to be connected, because it is not known whether $g_J$ is tame
or not for $J \in \CJ_{\omega}$.
However we will prove that the subset $\CJ_{\omega;\mathfrak{T}} \subset \CJ_{\omega}$
is contractible in the strong $C^r$ topology \emph{provided $3 \leq r < \infty$ 
for a given choice of quasi-isometry
class $\mathfrak{T}$ of a metric $g$ tame to $\omega$}.

\begin{theorem} \label{thm:omegaTJ-contractible}
Consider a smooth manifold $M$ and fix a quasi-isometry class $\mathfrak{T}$ 
of Riemannian metrics of $M$.
Then the set $\CJ_{\omega;\mathfrak{T}}$ of
quasi-isometrically tame almost complex structures is contractible with respect to
the $C^r$ topology with $3 \leq r < \infty$ 
of the space $\Gamma (TM)$ of sections of tangent bundles
induced from that of $(M,\mathfrak{T})$.
\end{theorem}
The rest of the section will be occupied by the proof of this theorem.

In the pointwise point of view at each $x\in M$ and $t\in [0,1]$,  
we want to find an endomorphism   $J_{x;t}: T_xM \to T_xM$ satisfying $J_{x;t}^2 = -\id$
that is compatible with $\omega_x$ over the family of the inner products $g_{x;t}$.
Solving this linear algebra problem relies on the well-known polar decompostion
in the linear algebra.

We start with recalling the polar decomposition, following the exposition of
\cite{daSilva}.
\subsection{Review of polar decomposition}

Let $(V,\Omega)$ be a symplectic vector space.
Let $G$ be any Euclidean inner product.
Nondegeneracy of $\Omega$ and $G$ determine two isomorphisms
$$
\alpha: v\mapsto \Omega(v,\cdot)
$$
$$
\beta:v \mapsto G(v,\cdot)
$$
from $V$ to its dual space $V^*$. Then $\alpha$ is a skew-symmetric and
$\beta$ is a symmetric linear map.
Then we obtain a unique endomorphism $A:=\beta^{-1}\circ\alpha$ on $V$
defined by the relation
\be\label{eq:Omega=GA}
\Omega(u,v)= G(A u,v).
\ee
We denote  by $A^*$ the adjoint linear map of $A \in \End(V)$ \emph{with respect to the inner product $G$}. Then
$$
G(A^*u,v)= G(u,Av)=G(Av,u)=\Omega(v,u)=-\Omega(u,v)=-G(Au,v)
$$
and $A^*=-A$, so $A$ is skew-symmetric with respect to $G$.
Then
\begin{itemize}
\item $AA^*$ is symmetric (with respect to $G$): $(AA^*)^*=AA^*$.
\item $AA^*$ is positive definite: $G(AA^*v,v)=G(A^*v,A^*v)>0$ for all $v\neq 0$.
\end{itemize}
Therefore, $AA^*$ is diagonalizable with positive eigenvalues $\lambda_i$, so that
$$
AA^*=BD B^{-1}
$$
for a diagonal matrix
\be\label{eq:Dt}
D = \text{\rm diag}\{\lambda_1,\dots,\lambda_{2n}\}
\ee
 with $\lambda_i > 0$.
Therefore we can define the square root of $AA^*$ by rescaling the eigenspaces and get
$$
\sqrt{AA^*}=B\sqrt{D} B^{-1}, \quad
\sqrt{D} = \text{\rm diag}\{\sqrt{\lambda_1},\dots,\sqrt{\lambda_{2n}}\}.
$$
Then $\sqrt{AA^*}$ is again symmetric and positive definite. If we put
\be\label{eq:J}
J:=(\sqrt{AA^*})^{-1}A
\ee
then $J^2 = -\text{\rm Id}$.
(The factorization $A=(\sqrt{AA^*})J$ is called the \emph{polar decomposition} of $A$.)

The following continuity is an immediate consequence of the uniqueness of the construction.
\begin{lem}\label{lem:gtoJ} Denote by $J_G$ the above $J$ associated to $G$ (and $\omega$). Then the map
$$
\text{\rm Sym}_+^2(V) \to \End(V); \quad G \mapsto J_G
$$
is continuous
where $\text{\rm Sym}_+^2(V)$ is the set of positive definite symmetric quadratic forms on $V$.
\end{lem}
\begin{proof} By definition of the matrix $A$ \eqref{eq:Omega=GA}, we may identify
$\text{\rm Sym}_+^2(V)$ with an open subset of symmetric matrices equipped with
the subspace topology of the latter which is homeomorphic to $\R^{n(n+1)/2}$.
Then we have the explicit expression of the map given by \eqref{eq:J}
from which continuity follows.
\end{proof}

\subsection{Uniform pinching estimates of eigenvalues of $A_tA_t^*$}
\label{subsec:pinching}

We fix a reference $(\omega,\mathfrak{T})$-tame almost complex structure
$J_0\in \CJ_{\omega;\mathfrak{T}}$.
Define $g_0:= \omega(\cdot, J_0 \cdot)\in \mathfrak{T}$.
By Theorem  \ref{thm:convexity}, the map
$$
L :[0,1] \times \text{\rm Riem}_\mathfrak{T}(M) \to \text{\rm Riem}_\mathfrak{T}(M)
$$
given by  the linear interpolation $\gamma(t,g)=(1-t)g_0 + t g$
is well-defined for $g\in \text{\rm Riem}_\mathfrak{T}(M)$ in that $g_t$
is contained in the same quasi-isometry class $\mathfrak T$.

Let $J_1$ be any element in $\CJ_{\omega;\mathfrak{T}}$ and
write $g_1:= \omega(\cdot, J_1 \cdot)\in \mathfrak{T}$. By definition,
$g_1$ is $C$-quasi-isometric to $g_0$ for some $C \geq 1$.
(Here $C = A$ in Definition \ref{defn:A-bilipschitz}.)
Then $t \mapsto L(t, g_0)$ gives us a path from $g_0$ to $g_1$. We denote
$$
g_t: = L(t,g_0)
$$
which is consistent with the notation for the initial condition $L(0,g_0) = g_0$.
We apply polar decomposition to each $g_{t,x}$ for $(t,x) \in [0,1] \times M$ and
obtain $A_{t,x} \in \End(T_xM)$. We write the associated bundle map by $A_t \in \End(TM)$.

By the unique algebraic process performed in the above
polar decomposition, the assignment $t \mapsto A_t$ defines a continuous path
in $C^\infty$ topology of $\End(TM)$. We then define
\be\label{eq:Jt}
J_{t,x}:=(\sqrt{A_{t,x}A_{t,x}^*})^{-1}A_{t,x}
\ee
by applying \eqref{eq:J} to each $t$ and $x$. We write by $J_t$ the associated
almost complex structure of $M$.

Since $A_t$ commutes with $\sqrt{A_tA_t^*}$, $J_t$ commutes with $\sqrt{A_tA_t^*}$.
Furthermore
\begin{itemize}
\item $J_t$ is skew-symmetric with respect to $g_t$;
$$
J_t^*=A_t^*(\sqrt{A_tA_t^*})^{-1}=-A_t(\sqrt{A_tA_t^*})^{-1}=-J_t.
$$
\item $J_t$ is orthogonal:
$J_t^*J_t=A_t^*(\sqrt{A_tA_t^*})^{-1}(\sqrt{A_tA_t^*})^{-1}A_t=\text{\rm Id}$.
\end{itemize}
We check that $J_t$ is compatible with $\omega$: we compute
\begin{itemize}
\item $\omega(J_tu,J_tv)=g_t(A_tJ_tu,J_tv)=g_t(J_tA_tu,J_tv)=g_t(A_tu,v)=\omega(u,v)$
\item $\omega(v,J_tv)=g_t(A_tv,J_tv)=g_t(-J_tA_tv,v)=g_t(\sqrt{A_tA_t^*}v,v)>0$ for all $v\neq 0$.
\end{itemize}
Hence we obtain a $J_t$-tame metric $g_{J_t}: = \omega(\cdot, J_t \cdot)$ for each $t \in [0,1]$.

Now we compare the $J_t$-tame metric
$g_{J_t}=\omega(\cdot,J_t\cdot)$ and $g_t$. \emph{Note that $g_{J_t}$ is constructed
via the two-step process
\be\label{eq:two-step}
g_t \mapsto J_t \mapsto g_{J_t}
\ee
where $g_{J_t}$  is not necessarily the same as the starting metric $g_t$.}
Using the definition, we compute
$$
g_{J_t}(u,v)=\omega(u,J_tv)=g_t(A_tu,J_tv)=g_t(J_t^*A_tu,v)=g_t(\sqrt{A_tA_t^*}u,v).
$$
Therefore we obtain the inequality
\be\label{eq:gJt-gt}
\left(\min_{1\leq i\leq 2n}{\sqrt{\lambda_{i,t}}}\right)g_t(u,u)\leq g_{J_t}(u,u)\leq
\left(\max_{1\leq i\leq 2n}{\sqrt{\lambda_{i,t}}}\right)g_t(u,u).
\ee

This shows that \emph{once we find a lower bound (away from zero) and an upper bound of $\sqrt{\lambda_{i,t}}$}, we can conclude that $g_{J_t}$ and $g_t$ are \emph{quasi-isometric}.

By now, we have reduced the contractibility proof to the study of
the uniform bounds for the eigenvalues $\lambda_{i,t}$
of the symmetric linear map $A_tA_t^*$ with respect to the metric $g_t$.
We then examine the bounds of the eigenvalues of $\lambda_{i,t}$ henceforth.

Recall that $\omega$ is independent of $t$ and $\beta_t=(1-t)\beta_0+t\beta_1$ by the definition of $\beta$
since $g_t=(1-t)g_0+tg_1$. Now we have the linearity of the map $t \mapsto A_t^{-1}$:
$$
  \begin{aligned}
A_t^{-1}&=\alpha^{-1}\circ \beta_t=\alpha^{-1}\circ((1-t)\beta_0+t\beta_1)\\
&=(1-t)\alpha^{-1}\circ\beta_0+t\alpha^{-1}\circ\beta_1=(1-t)A_0^{-1}+tA_1^{-1}.
  \end{aligned}
$$
Also note that $g_{J_0}=g_0,\, g_{J_1}=g_1$ which implies $A_0=J_0$ and $A_1=J_1$.
Therefore,
$$
A_t^{-1}=-(1-t)J_0-tJ_1.
$$
Since $A_tA_t^*$ is diagonalizable with positive eigenvalues,
$(A_tA_t^*)^{-1}=(A_t^*)^{-1}A_t^{-1}$ is also diagonalizable with positive eigenvalues and

\begin{eqnarray}\label{eq:A*t-1At-1}
 (A_t^*)^{-1}A_t^{-1}& =& - A_t^{-1}A_t^{-1}=-(-(1-t)J_0-tJ_1)^2 \nonumber\\
&= &-t^2J_1^2-t(1-t)J_0J_1-t(1-t)J_1J_0-(1-t)^2J_0^2 \nonumber\\
&=& (2t^2-2t+1)\text{\rm {Id}}-t(1-t)J_0J_1-t(1-t)J_1J_0 \nonumber\\
 &=& \text{\rm {Id}}+t(1-t)(-J_0J_1-J_1J_0-2\text{\rm {Id}}).
\end{eqnarray}

\begin{lem}\label{lem:inverse}
We have $(-J_1 J_0)^{-1}=- J_0 J_1$.
\end{lem}
\begin{proof} We compute
\begin{equation}\label{eqn:riem_met_comp}
  \begin{aligned}
g_0(u,v)=\omega(u,J_0v)=\omega(u,-J_1J_1J_0v)=g_1(u,(-J_1J_0)v),\\
g_1(u,v)=\omega(u,J_1v)=\omega(u,-J_0J_0J_1v)=g_0(u,(-J_0J_1)v).
  \end{aligned}
\end{equation}
The lemma immediately follows from this.
\end{proof}

Postponing the discussion of the boundary case $t = 0, \, 1$ till the end of the proof,
we first consider the case $t\in (0,1)$.
For this purpose, the following simple
result in linear algebra plays an important role.

\begin{lem}\label{lem:curious}
Let $M$ be an invertible real matrix.
Then every eigenvalue of $M+M^{-1}$ is of the form $\lambda_i+\frac{1}{\lambda_i}$,
where $\lambda_i$ is an eigenvalue of $M$.
\end{lem}
\begin{proof} This follows from the fact that $M$ and $M^{-1}$ commute each other and that
for such a pair one can obtain \emph{simultaneous} Jordan canonical forms. For completeness'
sake, we provide the details in an Appendix \ref{sec:Jordan}.
\end{proof}

We derive
\be\label{eq:-J1J0-J0J1}
-J_1 J_0-J_0 J_1=2\text{\rm {Id}}+\frac{1}{t(1-t)}((A_t^*)^{-1}A_t^{-1}-\text{\rm {Id}})
\ee
from \eqref{eq:A*t-1At-1}
which shows that $-J_1 J_0-J_0 J_1$ is symmetric with respect to the metric $g_t$
for all $0 < t < 1$. Therefore all eigenvalues are real.
We also obtain
$$
-J_1 J_0 -J_0 J_1 = -J_1 J_0+(-J_1 J_0)^{-1}
$$
from Lemma \ref{lem:inverse}.

Now let $u_i$ be an eigenvector of $-J_1J_0$ with eigenvalue $\lambda_i$.
Then $u_i$ is also an eigenvector of $(A_t^*)^{-1}A_t^{-1}$ with eigenvalue
\be\label{eq:1/lambdait}
\frac{1}{\lambda_{i,t}} : =1+t(1-t)\left(\lambda_i+\frac{1}{\lambda_i}-2\right)
\ee
where $\lambda_{i,t}$ is the associated eigenvalue of $A_tA_t^*$.

We note that $\lambda_i$ is not zero and may not necessarily be real since $-J_1J_0$ is
not necessarily symmetric. We also know that
\begin{itemize}
\item the eigenvalue $\lambda_i$ gives rise to an eigenvalue of $-J_1 J_0-J_0 J_1$
given by $\lambda_i+\frac{1}{\lambda_i}$, and  they are real by the symmetry thereof.
\item every eigenvalue
of $-J_1 J_0-J_0 J_1$ is of this form by Lemma \ref{lem:curious}.
\end{itemize}
We note that the reality in the former statement is possible \emph{only when either
$\lambda_i$ is real or it satisfies $|\lambda_i|=1$.}

We consider the two cases separately.

\medskip

\noindent{\bf Case 1: $\lambda_i$ is real.}
\smallskip

If $u_i$ is an eigenvector of $-J_1J_0$ with real eigenvalue $\lambda_i$,
$$
g_0(u_i,u_i)=g_1(u_i,-J_1J_0u_i)=g_1(u_i,\lambda_i u_i)=\lambda_ig_1(u_i,u_i).
$$
Since $g_0$ and $g_1$ are assumed to be $C$-quasi-isometric in the very beginning of the current
subsection, it follows 
that there exists a constant $ C \geq 1$ ($C = A$ in Definition \ref{defn:A-bilipschitz})
such that
$$
\frac{1}{C}g_1(u_i,u_i)\leq g_0(u_i,u_i)\leq C g_1(u_i,u_i).
$$
This implies $\frac{1}{C}\leq \lambda_i \leq C$ and we obtain
$$
2\leq \lambda_i+\frac{1}{\lambda_i}\leq C+\frac{1}{C}.
$$
We dereve the uniform bound
\be\label{eq:lambdabound1}
1\leq \frac{1}{\lambda_{i,t}}\leq \frac{1}{2}+\frac{1}{4}\left(C+\frac{1}{C}\right)
\ee
from \eqref{eq:1/lambdait}
for all $t \in (0,1)$ holds in this case.

\medskip

\noindent{\bf Case 2:
$\lambda_i$ is not real and $|\lambda_i|=1$.}
\smallskip

In this case, we have $\frac{1}{\lambda_i}=\overline{\lambda_i}$ and $\lambda_i+\frac{1}{\lambda_i}$
is a real number since the latter is an eigenvalue of the symmetric endomorphism
$-J_1J_0 - J_0J_1$.  This in particular implies
$-2< \lambda_i+\frac{1}{\lambda_i}< 2$.
We will now prove the following inequality.
\begin{lem}\label{lem:positivity} Assume Case 2. Then for all $i$, we have
\be\label{eq:>0}
\lambda_i+\frac{1}{\lambda_i} \geq  0
\ee
\end{lem}
\begin{proof} We prove this  by contradiction.

Suppose to the contrary that there is
 an eigenvector $u_i$ of $-J_1J_0$ with eigenvalue $\lambda_i$ such that \eqref{eq:>0}
 fails to hold so that $\lambda_i+\frac{1}{\lambda_i} < 0$. Then
$$
g_0(u_i,(-J_1J_0-J_0J_1)u_i)=g_0\left(u_i,\left(\lambda_i+\frac{1}{\lambda_i}\right)u_i\right)
= \left(\lambda_i+\frac{1}{\lambda_i}\right)g_0(u_i,u_i)<0
$$
by the standing hypothesis and the positivity $g_0(u_i,u_i) > 0$.
On the other hand, we also have
\beastar
g_0(u_i,(-J_1J_0-J_0J_1)u_i)& = &g_0(u_i,-J_1J_0u_i)+g_0(u_i,-J_0J_1u_i)\\
&= &g_1(u_i,(-J_1J_0)^2 u_i)+g_1(u_i,u_i)\\
&= &(\lambda_i)^2g_1(u_i,u_i)+g_1(u_i,u_i)\\
&= &((\lambda_i)^2+1)g_1(u_i,u_i).
\eeastar
Therefore we obtain
$$
0 > \left(\lambda_i+\frac{1}{\lambda_i}\right)g_0(u_i,u_i) = (\lambda_i^2+1)g_1(u_i,u_i).
$$
We derive that $\lambda_i^2+1$ is also real and satisfies
\be\label{eq:lambdai2+1}
\lambda_i^2+1 < 0
\ee
since $g_1(u_i, u_i) > 0$.
In particular $\lambda_i^2$ is a real number with $|\lambda_i| = 1$, but 
$\lambda_i$ itself is not real by the standing assumption of the current case. 
This implies $\lambda_i = \pm \sqrt{-1}$, and hence $\lambda_i^2+1 = 0$, 
which contradicts to \eqref{eq:lambdai2+1}. 
This finishes the proof of \eqref{eq:>0}.
\end{proof}

\begin{cor}
 \be\label{eq:lambdabound2}
 \frac{1}{2}\leq \frac{1}{\lambda_{i,t}} <  1.
 \ee
 for all $ t \in (0,1)$ in this case.
 \end{cor}
 \begin{proof} Recall the formula \eqref{eq:1/lambdait}
 $$
 \frac{1}{\lambda_{i,t}} : =1+t(1-t)\left(\lambda_i+\frac{1}{\lambda_i}-2\right).
 $$
 The upper bound immediately follows which corresponds to the case 
 $\lambda_i+\frac{1}{\lambda_i} \to 2$.
 On the other hand, when $\lambda_i+\frac{1}{\lambda_i} = 0$,
 we have
 $$
 \frac{1}{\lambda_{i,t}} = 1 -2t(1-t) = 2t^2 - 2t +1 \geq \frac12
 $$
 where the minimum is achieved at $t = \frac12$.
 This finishes the proof.
 \end{proof} 
  
We summarize the above discussion into the following from the above consideration of
the two cases above
\begin{prop}\label{prop:pinching-estimate}
Let $g_0$, $g_1$ and $J_0$, $J_1$ be as above, and consider the convex sum
$g_t = (1-t)g_0 + tg_1$. Then
$$
\frac12 \leq \frac1{\lambda_{i,t}} \leq  \frac12 + \frac14 \left(C + \frac1C\right)
$$
for all $i$ and $0 \leq t \leq 1$.  Furthermore the constant $C$
depends only on $g_0, \, g_1$.
\end{prop}
\begin{proof} First consider the case $0 < t <1$. By combining \eqref{eq:lambdabound1}, \eqref{eq:lambdabound2},
the proposition follows for $0 < t < 1$.
Then by continuity the same bounds also hold for $t = 0, \, 1$.
\end{proof}

\begin{rem} \begin{enumerate}
\item There is one interesting point hidden in the above proof of symmetry of the
endomorphism $-J_1 J_0-J_0 J_1$: there were no simple a priori reason for this symmetry
to hold with respect the metric $g_t$ for $0 < t < 1$ without the identity \eqref{eq:A*t-1At-1}.
Furthermore as we apply the
continuity argument from $0 < t < 1$ to the boundary points $t=0, \, 1$, there is
no a priori reason why the right hand side of
the equation \eqref{eq:-J1J0-J0J1} is continuous at $t = 0, \, 1$ either
unless we have the identity \eqref{eq:A*t-1At-1} which shows that
 it is a constant function over $t \in [0,1]$. In fact,
we can also verify the identity by directly computing
\beastar
&{}& \frac{1}{t(1-t)}((A_t^*)^{-1}A_t^{-1}-\text{\rm {Id}})\\
 & = & -J_1 J_0-J_0 J_1 -
\frac{(1-t)}t J_0^2 - \frac{t}{1-t} J_1^2 -\frac{\text{\rm {Id}}}{t(1-t)} \\
& = & -J_1 J_0-J_0 J_1-2\text{\rm {Id}}.
\eeastar
The apparent singular behavior at $t = 0, \, 1$ of the second expression
disappears by the fact that $J_i$  satisfy the equation $J_i^2 = -\text{\rm Id}$ for both $i = 0, \, 1$
and the equality
$$
\frac{(1-t)}t  + \frac{t}{1-t} -\frac{1}{t(1-t)} = -2.
$$
\item Indeed the above proof seems to display some curious linear algebraic
interaction between the symplectic form and its compatible almost complex
structures, which we think deserves further underpinning.
\end{enumerate}
\end{rem}

\subsection{Tameness of the metric $g_{J_t}$}
\label{subsec:gJ-tame}

To conclude that $J_t$ is $(\omega,\mathfrak{T})$-tame and 
$J_t\in \CJ_{\omega;\mathfrak{T}}$ for $t\in(0,1)$, 
we have to show that $g_{J_t}$ is tame, i.e., complete and of bounded $\|R_{g_{J_t}}\|_{C^1}$
and has a uniform injectivity  lower bound for $t \in [0,1]$.
We have shown that $g_t$ and $g_{J_t}$ are quasi-isometric (and so bilipschitz) 
in the last subsection, and $g_{J_t}$ is also tame by 
Theorem \ref{thm:no-cusp}.

\section{Quasi-isometric equivalence and contractibility of $\CJ_{\omega;\mathfrak{T}}$}

We now go back to the proof of contractibility of $\CJ_{\omega;\mathfrak{T}}$ and
are ready to wrap-up the proof of  Theorem \ref{thm:omegaTJ-contractible}.

Let $J_0$ be a given $\omega$-tame almost complex structure and $J_1$
be a general element of $\CJ_{\omega;\mathfrak{T}}$. We consider the associated metric
$g_{J_i}$ for $i=0, \, 1$. Since both $J_0, \, J_1$ are in the same quasi-isometry class
$\CJ_{\omega;\mathfrak{T}}$, there is a constant $C \geq 1$ such that
$g_1$ is $C$-quasi-isometric with $g_0$.

We consider the map
$$
\CG: \CJ_{\omega;\mathfrak T} \to \text{\rm Riem}_{\mathfrak{T}}(M)
$$
given by $\CG(J): = g_J = \omega(\cdot, J \cdot)$ which is well-defined by
the definition of $\CJ_{\omega;\mathfrak T}$.
We also denote by
$$
\CP: \text{\rm Riem}_{\mathfrak{T}}(M) \to \CJ_{\omega;\mathfrak{T}}; \quad g \mapsto J_g
$$
Both maps are continuous in strong $C^r$ topology. The compoisition
$\CG \circ \CP$ is nothing but the map obtained by the aforementioned two-step 
process \eqref{eq:two-step} which is also continuous in strong $C^r$ topology.

Consider the subset
$$
\text{\rm Riem}_{\mathfrak{T};\omega}(M): = \{ g \in \text{\rm Riem}_{\mathfrak{T}}(M)
\mid g = g_J, \, J \in \CJ_{\omega;\mathfrak{T}} \}.
$$
The following property of composition plays a fundamental role in our proof of
contractibility of $\CJ_{\omega;\mathfrak{T}}$.

\begin{lem}\label{lem:PG-retracts}
The image of composition $\CG \circ \CP$ is a retraction of
$$
\text{\rm Riem}_{\mathfrak{T};\omega}(M)  \subset \text{\rm Riem}_{\mathfrak{T}}(M),
$$
i.e.,
\be\label{eq:retraction}
\CG \circ \CP|_{\text{\rm Riem}_{\mathfrak{T};\omega}(M) }
= \text{\rm Id}|_{\text{\rm Riem}_{\mathfrak{T};\omega}(M) }.
\ee
\end{lem}
\begin{proof} This is an immediate consequence of
 the uniqueness of polar decomposition: See \eqref{eq:Omega=GA}.
 More specifically, we have
  $$
  \CG(\CP(g_J))=\CG(J_{g_J})=\omega(\cdot, J_{g_J} \cdot)=g_J
  $$
  for all $J \in \CJ_\omega$, and hence the identity \eqref{eq:retraction}.
 \end{proof}
 
 \begin{rem}
 As we mentioned before, $\CG(\CP(g))$ is not necessarily the same as $g$
if $g \in \text{\rm Riem}_{\mathfrak{T}}(M) \setminus \text{\rm Riem}_{\mathfrak{T};\omega}(M)$. 
\end{rem}

Going back to the proof of Theorem \ref{thm:omegaTJ-contractible},
we denote by
$$
\CH_{g_0}: [0,1] \times \text{\rm Riem}_{\mathfrak{T}}(M) \to \text{\rm Riem}_{\mathfrak{T}}(M)
$$
the above contraction $(t,g) \mapsto (1-t)g_0 + tg$ which is
$C_g$-quasi-isometric by Theorem \ref{thm:convexity} for a constant $C_g \geq 1$
(depending on $g$), and hence the homotopy is well-defined and continuous
in $C^\infty$ topology.

Then to construct the required contraction, we have only to take the composition
$$
\CP \circ \CH_{g_0}\circ (\text{\rm Id}_{[0,1]} \times \CG)
:[0,1] \times \CJ_{\omega;\mathfrak{T}} \to \CJ_{\omega;\mathfrak{T}}
$$
which defines the required contraction homotopy of $\CJ_{\omega;\mathfrak{T}}$.
This finally finishes the proof of  Theorem \ref{thm:omegaTJ-contractible}.
\qed

The following theorem is an immediate consequence of the standard
practice in symplectic topology via the machinery of pseudo-holomorphic curves.

\begin{theorem} Denote by $(M,\omega; \mathfrak{T})$ a (noncompact)
smooth symplectic manifold $M$ with a given quasi-isometry class $\mathfrak{T}$ of 
Riemannian metrics on $M$. Let $\aleph(\omega,\mathfrak{T};J)$ be an invariant 
constructed via $J$-holomorphic curves with $J \in \CJ_{\omega;\mathfrak{T}}$. 
Then it does not depend on the choice of such $J$'s. 
More precisely, we have
$$
\aleph(\omega,\mathfrak{T};J) = \aleph(\omega,\mathfrak{T}
;\phi^*J).
$$
for any symplectic diffeomorphism $\phi:(M,\omega) \to (M,\omega)$ that preserves
$\mathfrak{T}$, i.e., for those that satisfies
$\phi^*J \in \CJ_{\omega;\mathfrak{T}}$ whenever $J \in \CJ_{\omega;\mathfrak{T}}$.
\end{theorem}

This leads us to consider the following subgroup of $\text{\rm Symp}(M,\omega)$
similarly as in \cite{choi-oh:construction}.

\begin{defn}
 We denote by
$$
\text{\rm Diff}_{\text{\rm Lip}}(M,\mathfrak{T})
$$
the automorphism group of $\mathfrak{T}$. We then define the automorphism group
of $(\omega;\mathfrak{T})$ to be the intersection
$$
\Symp(M,\omega) \cap\text{\rm Diff}_{\text{\rm Lip}}(M,\mathfrak{T})
=: \Symp_{\text{\rm Lip}}(M,\omega;\mathfrak{T}).
$$
\end{defn}

\begin{rem}
In the present paper, we have studied the Lipschitz topology in the point of view of
\emph{large-scale} symplectic topology. We can also contemplate the Lipschitz topology
in the point of view of \emph{micro-scale} or \emph{PL} topology and low regularity
symplectic or contact topology. (See \cite{kleiner-muller-xie} and references therein
for some studies of bilipschitzian contact invariants.) We will come back to
the study of the micro-scale symplectic topology elsewhere.
\end{rem}

\section{The case of 2 dimensional Riemann surfaces}

In \cite{choi-oh:construction}, the authors constructed a Fukaya category
$Fuk(M,\CT)$ associated to each hyperbolic structure $\CT$ on a Riemann surface,
especially of infinite type. Denote by $\mathfrak T$ the quasi-isometry class of the
hyperbolic structure $\CT$.
It is stated that the $A_\infty$ category is quasi-equivalent
when one deforms almost complex structures tame to $\omega$ varies inside the
set $\CJ(\CT)$ consisting of \emph{$\CT$-tame} almost complex structures. 

We fix  the reference hyperbolic structure on $M$ in the sense of \cite{liu-papado}.

\begin{defn}[Hyperbolic Riemann surface]\label{defn:hyperbolic-structure-intro}
 A hyperbolic Riemann surface
is a triple $(\Sigma,J_0, g_0)$ whose universal cover is isometric to the unit disk.
We call it \emph{tame} if it has bounded curvature and its injectivity radius
is positive.
\begin{enumerate}
\item A \emph{hyperbolic structure}, denoted by $\CT = \CT_\Sigma$,
of a surface $\Sigma$ is a choice of  $(\Sigma,J_0, g_0): = (\Sigma,J_\CT,g_\CT)$
that is tame.
\item
$\CT$ also determines a symplectic form
$$
\omega_\CT = g_0(J_0 \cdot, \cdot)
$$
which we call $\CT$-symplectic form.
\end{enumerate}
\end{defn}

\begin{defn}[$\omega_\CT$-tame almost complex structures]
 We call $J$ an $\omega_\CT$-tame almost complex structure
 if the following holds:
 \begin{enumerate}
 \item It is $\omega_\CT$-tame.
 \item The metric $g_J = \omega_\CT(\cdot, J \cdot)$ is
 quasi-isometric to $g_\CT = \omega(\cdot, J_\CT \cdot)$.
 \end{enumerate}
We denote by $\CJ(\CT)$ the set of $\omega_\CT$-tame almost
 complex structures.
 \end{defn}

An immediate corollary of Theorems \ref{thm:convexity} and \ref{thm:omegaTJ-contractible}
 is the following special case of two-dimensions. 
 
\begin{theorem} \label{thm:surface-case} 
Let $\Sigma$ be a noncompact surface equipped with hyperbolic structure.
Denote by $\mathfrak T$ be a quasi-isometry class of hyperbolic structures of $\Sigma$ and by $\Riem_\mathfrak T(\Sigma)$
the set of Riemannian metrics quasi-isometric to the given hyperbolic metric. Then
the set $\Riem_\mathfrak T(\Sigma)$ is contractible in  strong $C^r$ topology with
$3 \leq r < \infty$.
\end{theorem}

\begin{rem}
This theorem in turn is used by the authors of \cite{choi-oh:construction} therein
for the construction of a Fukaya category of
a surface of infinite type $Fuk(\Sigma,\CT)$  as a \emph{quasi-isometric symplectic invariants}
which a priori depends on the choice of quasi-isometry class $\mathfrak T$ of the hyperbolic structure of
the Riemann surface $\Sigma$.
\end{rem}

\appendix

\section{ Proof of Lemma \ref{lem:IFT-condition} }
\label{sec:IFT-condition}

Recall the definition 
$$
F_i(x): = (\psi_p^{g_0})^{-1}( \psi_p^{g_{s_i}} (x))
$$ 
from \eqref{eq:Fi} which are  well-defined as  maps
$$
F_i: B^n(r_0/2) \to B^n(r_0)
$$
with fixed domain and codomain. 

For the proof of Lemma \ref{lem:IFT-condition}, it is enough to prove the following estimate of
the derivative of the Jacobian $\left(\frac{\del y_i}{\del x_j}\right)$.

\begin{lem}[Compare with Lemma 4.3 \cite{cheeger:finiteness}]\label{lem:2nd-derivative}
Given $S, \, S_1 > 0$, there exists a
constant $C(S,S_1)$ such that if $r_0 < \frac{\pi}{2 \sqrt{S}}$ and 
$\|R_i\|_{C^0} < S, \, \|DR_i\|_{C^0} < S_1 $, then on $B^n(r) \subset \R^n$, we have
\be\label{eq:D2Fi}
\|D^2 F_i\|_{C^0} \leq C(S,S_1).
\ee
\end{lem}
One essential difference between the framework of \cite[Lemma 4.3]{cheeger:finiteness}
and that of Lemma \ref{lem:2nd-derivative} is that while the former deals with coordinate
change maps between two normal coordinates of \emph{the same metric}, the latter 
involves normal coordinates of two different metrics \emph{centered at a sequence of 
the same point}. The main interest of our study is the case when the sequence
escapes to infinity on $M$. The linear isometry $I_p^g$ is used to make the center
of the balls fixed at the origin in $\R^n$ so that we can compare the 
exponential maps of two different metrics $g_0$ and $g_{s_i}$.

Let $(x_1, \cdots, x_n)$ and $(y_1, \cdots, y_n)$ be normal coordinate systems 
of  $g_{s_i}$ and of $g_0$ on  $B_{r_0/2}(p_i)$ and $B_{r_0}(p_i)$ 
 based on frames  $\{e_i\}_{p_i }^{g_{s_i}}, \, \{f_i\}_{p_i}^{g_0}$, respectively.
 Through the linear isometries $I_{g_{s_i}}^{p_i}$ and $I_{g_0}^{p_i}$, we may
 safely assume that the coordinate systems $(x_1, \cdots, x_n)$ and $(y_1, \cdots, y_n)$
 are defined on $\R^n$ by identifying $y_i$ with $y_i \circ I_{g_0}^{p_i}$ and
 $x_i$ with $x_i \circ I_{g_{s_i}}^{p_i}$. After this identification, 
 \eqref{eq:D2Fi} is equivalent to the bound
 \be\label{eq:2nd-derivative-bound}
 \left| \frac{\del^2 y_i}{\del x_k\del x_j} \right| \leq C(S,S_1).
 \ee
 Then using Lemma \ref{lem:cheeger}, it will suffice to have an upper bound for
 the covariant derivatives
$$
 \left\| \nabla_{\frac{\del}{\del x_j}}\right\|, \,  \left\|\nabla_{\frac{\del}{\del y_k}}  \right\|
 $$
 in terms of $S, \, S_1$. With these preparations, Cheeger's proof  of 
 \cite[Lemma 4.3]{cheeger:finiteness} verbatim applies without change.
 This finishes the proof.
 
\section{Direct limit strong $C^\infty$ topology}
\label{sec:direct-limit}

In this appendix, we introduce the notion of 
\emph{direct limit strong $C^\infty$ topology}, which is
weaker that the usual definition of strong $C^\infty$ topology but 
stronger than the weak $C^\infty$ topology, as follows. 

Consider any tensor bundle $\CT \to M$ over a noncompact
Riemannian manifold $(M,g)$ equipped with a tame metric $g$. 

We consider the space $C^\infty(\CT)$ of smooth sections of $\CT$.
Let us fix an atlas of $M$ of the type $\{\B_\alpha(r)\}$ of geodesic
balls of radius $r > 0$ with $r < \iota_g$.  Then consider the filtration of 
$C^\infty(\CT)$ given by
\be\label{eq:Cinfty-epsilon}
C^\infty_{N}(\CT): = \{T \in \Gamma(\CT) \mid \|D^rT\| < N\,  \forall r \in \N\}
\ee
equipped with the subspace topology of the strong $C^\infty$ topology. Then we have
$$
C^\infty(\CT) = \bigcup_{N \in \N} C^\infty_{N}(\CT)
$$ 
as an increasing union. We call the direct limit topology of the directed system
$$
C^\infty_{1}(\CT) \hookrightarrow C^\infty_{2}(\CT) \hookrightarrow \cdots \hookrightarrow C^\infty_{N}(\CT)
\cdots  \longrightarrow C^\infty(\CT)
$$
the \emph{direct limit strong $C^\infty$ topology}. It is easy to proof the following contractibility
by a slight modification of the proof of Theorem \ref{thm:Cr-contractibility}.

\begin{theorem}\label{thm:Cinfty-contracibility}
The path $s \mapsto (1-s) g + s g_{\text{\rm rf}}$ is continuous
with respect to the direct limit $C^\infty$ topology.
\end{theorem}

\section{Proof of Lemma \ref{lem:curious}}
\label{sec:Jordan}

In this appendix, we give the details of the proof of Lemma \ref{lem:curious}.

Let $M$ be an invertible real matrix.
Note that every eigenvector of $M$ is also an eigenvector of $M+M^{-1}$,
but the converse might not hold in general.

Consider a Jordan decomposition of $M$.
Then there exists an upper triangular matrix $U$, which is called Jordan normal form,
and an invertible matrix $P$ such that
$$
M=PUP^{-1}.
$$
$U$ is a block diagonal matrix
$$
U=\begin{bmatrix}
U_1 &&\\
&\ddots&\\
&&U_k\\
\end{bmatrix}
$$
Each $U_i$ is a Jordan block of the form
$$
U_i=\begin{bmatrix}
\lambda_i &1&&\\
&\lambda_i&\ddots&&\\
&&\ddots&1\\
&&&\lambda_i\\
\end{bmatrix}
$$
Where $\lambda_i$ is an eigenvalue of $M$ which is not zero
since $M$ is invertible. Also note that
$$
M^{-1}=PU^{-1}P^{-1}.
$$
Since $U$ is an invertible block diagonal matrix,
$$
U^{-1}=\begin{bmatrix}
U_1^{-1} &&\\
&\ddots&\\
&&U_k^{-1}\\
\end{bmatrix}
$$
Also, each $U_i$ is an upper triangular matrix and the inverse of an upper triangular
matrix is upper triangular implies $U_i^{-1}$ is also upper triangular.
Since every diagonal entry of $U_i$ is $\lambda_i$ which is not zero,
every diagonal entry of $U_i^{-1}$ is $\frac{1}{\lambda_i}$.

Now we have
$$
\begin{aligned}
M+M^{-1}&=PUP^{-1}+PU^{-1}P^{-1}\\
&=P(U+U^{-1})P^{-1}.
\end{aligned}
$$
and $U+U^{-1}$ is a block diagonal matrix and each block is an upper triangular
matrix whose diagonal entries are of the form $\lambda_i+\frac{1}{\lambda_i}$.
Therefore, they are eigenvalues of $M+M^{-1}$.
This implies that every eigenvalue of $M+M^{-1}$ is of the form
$\lambda_i+\frac{1}{\lambda_i}$ and for every such eigenvalue,
there exists an eigenvector $u_i$ of $M$ with eigenvalue $\lambda_i$ such that
$$
(M+M^{-1})u_i=\left(\lambda_i+\frac{1}{\lambda_i}\right)u_i.
$$
This finishes the proof.

\def\cprime{$'$}
\providecommand{\bysame}{\leavevmode\hbox to3em{\hrulefill}\thinspace}
\providecommand{\MR}{\relax\ifhmode\unskip\space\fi MR }
\providecommand{\MRhref}[2]{%
  \href{http://www.ams.org/mathscinet-getitem?mr=#1}{#2}
}
\providecommand{\href}[2]{#2}

\end{document}